\documentclass{amsart}

\title[Signatures of certain representations]{Signatures of representations of Hecke algebras and rational Cherednik algebras}
\author{Vidya Venkateswaran}
\address{Department of Mathematics, MIT, Cambridge, MA 02139}
\email{\href{mailto:vidyav@math.mit.edu}{vidyav@math.mit.edu}}
\thanks{Research supported by NSF Mathematical Sciences Postdoctoral Research Fellowship DMS-1204900}
\subjclass[2000]{}
\usepackage{hyperref}
\usepackage[svgnames]{xcolor}
\hypersetup{colorlinks,breaklinks,
            linkcolor=DarkBlue,urlcolor=DarkBlue,
            anchorcolor=DarkBlue,citecolor=DarkBlue}
\usepackage{url}
\usepackage{euscript}
\usepackage[verbose,letterpaper,tmargin=1in,bmargin=1in,lmargin=1.25in,rmargin=1.25in]{geometry}
\usepackage{amssymb}
\usepackage{amsmath}

\usepackage{xypic}
\usepackage{verbatim}

\newtheorem{theorem}{Theorem}[section]

\newtheorem{lemma}[theorem]{Lemma}
\newtheorem{proposition}[theorem]{Proposition}
\newtheorem{corollary}[theorem]{Corollary}

\theoremstyle{definition}
\newtheorem{definition}[theorem]{Definition}

\theoremstyle{remark}
\newtheorem*{remarks}{Remarks}

\begin{document}

\begin{abstract}
Determining whether an irreducible representation of a group (or $*$-algebra) admits a non-degenerate invariant, positive-definite Hermitian form is an important problem in representation theory.  In this paper, we study a related notion: that of signatures.  We study representations $S^{\lambda}(q)$ of $\mathcal{H}_{n}(q)$, the Hecke algebra of type $A$ ($|q| = 1$), and representations $M_{c}(\lambda)$ of $\mathbb{H}_{c}$, the rational Cherednik algebra of type $A$ ($c \in \mathbb{R}$), which have unique (up to scaling) invariant Hermitian forms (here $\lambda$ is a partition of $n$).  The signature is the number of elements with positive norm minus the number of elements with negative norm, and we analogously define the signature character in the case that there is a natural grading on the module.  We provide formulas for (1) signatures of modules over $\mathcal{H}_{n}(q)$ and (2) signature characters of modules over $\mathbb{H}_{c}$.  We study the limit $c \rightarrow -\infty$, in which case the signature character has a simpler form in terms of inversions and descents of permutations in $S(n)$.   We provide examples corresponding to some special shapes, and small values of $n$.  Finally, when $q = e^{2 \pi i c}$, we show that the asymptotic signature character of the $\mathbb{H}_{c}$-module $M_{c}(\tau)$ is the signature of the $\mathcal{H}_{n}(q)$-module $S^{\tau}(q)$.  

\end{abstract}

\maketitle

\section{Introduction}

Let $\mathcal{G}$ be an algebraic group and $V$ an irreducible complex representation of $G$.  Determining (1) if $V$ admits a non-degenerate invariant Hermitian form and (2) if this form is positive-definite (i.e., if the representation is unitary) is an important and often challenging problem in representation theory.  A similar problem can be phrased for algebras as well.  Unitary representations have been determined in many cases, for example see \cite{ESG} and the references therein.  

A refinement of unitarity may be found in the notion of \textit{signatures}, which we will now describe.  Suppose $W$ is a finite-dimensional vector space with a Hermitian form $\langle \cdot, \cdot \rangle$.  Let $\{e_{i}\}$ be a basis for $W$.  Recall that the Gram matrix $G$ has $(i,j)$ entry equal to $\langle e_{i}, e_{j} \rangle$.  The signature of $W$ is defined as the matrix signature of $G$, i.e., the number of positive eigenvalues minus the number of negative eigenvalues of $G$, which does depend on the choice of basis $\{e_i\}$.  Note that in the special case when $\{e_{i} \}$ is an orthogonal basis, the signature of $W$ is the number of basis elements with positive norm minus the number of basis elements with negative norm.

 In the situation of the first paragraph, suppose $V$ admits a non-degenerate invariant Hermitian form.  If $V$ is finite-dimensional, we may consider the signature of $V$ --- we denote this invariant by $s(V)$.  Thus, if the representation is unitary, the signature is the dimension of the representation.  If $V$ (now possibly infinite-dimensional) has a natural grading into finite-dimensional weight spaces which are orthogonal with respect to the form, we may instead define the \textit{signature character} as follows
\begin{equation*}
ch_{s}(V) = \sum_{w} t^{w} s(V_{w}),
\end{equation*}
where the sum is over degrees $w$ of $V$ and $V_{w}$ is the corresponding finite-dimensional weight space.  At $t=1$, one recovers the signature in the case that the representation is finite-dimensional.  If the representation is unitary, one recovers the Hilbert series with respect to this grading.  In this paper, we will investigate (1) signatures of representations of the Hecke algebra of type $A$, (2) signature characters of representations of the rational Cherednik algebra of type $A$, and (3) the relationship between them.  In our computations, we will assume $q$ is not a root of unity and $c$ is not of the form $r/m$ for $(r,m) = 1$ and $m = 2, \dots, n$.  We will look at $q$ and $c$ in the intervals formed by excluding these points of possible degeneracy.  

Let $\mathcal{H}_{n}(q)$ denote the Hecke algebra of type $A$ with parameter $|q| = 1$ (recall that one can define a Hecke algebra $\mathcal{H}_{\underline{q}}(W)$ over $\mathbb{C}$ associated to a complex reflection group $W$; we study the case $W = S_{n}$).  Much is known about representations of $\mathcal{H}_{n}(q)$ \cite{M, DJ1, DJ2}; we briefly summarize some of these results.  First, they are indexed by partitions $\lambda$ of weight $n$ (recall that a partition $\lambda$ is a string of non-increasing integers and the weight, denoted $|\lambda|$, is the sum of the parts).  Moreover, such a representation $S^{\lambda}(q)$ has a basis consisting of standard Young tableaux of shape $\lambda$.  The action of operators $T_{i} \in \mathcal{H}_{n}(q)$ can be described in terms of operations on these standard Young tableaux (with coefficients depending on $q$).  By \cite{DJ1, DJ2}, there is a symmetric bilinear form on $S^{\lambda}(q)$.  In \cite{St}, this is used to produce a Hermitian form on $S^{\lambda}(q)$ with the defining property of invariance under the braid group (i.e., $(Tv, v') = (v, T^{-1}v')$ for all $T \in B_{W} \subset W$).  Stoica also obtained a complete classification of unitary irreducible representations of $\mathcal{H}_{n}(q)$.  

One may start with a distinguished basis element of $S^{\lambda}(q)$ and construct a combinatorial algorithm that produces an arbitrary basis element via a series of applications of operators $T_{i}$.  Such an algorithm allows us to keep track of the way that the norm changes.  We use this to explicitly compute a formula for the signature of $S^{\lambda}(q)$ in Section 2:
\begin{theorem} \label{Heckesig}
Let $\lambda$ be a partition such that $|\lambda| = n$.  We have the formula for the signature of $S^{\lambda}(q^{2})$
\begin{equation*}
s(S^{\lambda}(q^{2})) = \sum_{\substack{T \in Std(\lambda) \\ (d_{1}, \dots, d_{n}) \text{ content vector of } T}} \prod_{\substack{1 \leq l < i \leq n \\ d_{i}-d_{l} < 0}} \{[[d_{i}-d_{l}+1]]_{q}\} \{[[d_{i}-d_{l}-1]]_{q}\},
\end{equation*}
where
\begin{equation*}
[[m]]_{q} = \frac{q^{m} - q^{-m}}{q-q^{-1}} \in \mathbb{R}
\end{equation*}
and $\{ \cdot \}: \mathbb{R} \setminus 0 \rightarrow  \pm 1 $ is the sign map.
\end{theorem}
\noindent We use the general formula above to compute signatures for representations corresponding to specific shapes.  For example, we provide simplified formulas for the hook and two-row shapes.  

In the next part of the paper, we consider the rational Cherednik algebra of type $A$. Recall that a rational Cherednik algebra $H_{c}(W, \mathfrak{h})$ is defined by a finite group $W$, a finite dimensional complex representation $\mathfrak{h}$ of $W$, and a function $c$ on conjugacy classes of reflections in $W$.  We are concerned with the rational Cherednik algebra of type $A$, which we denote by $\mathbb{H}_{c}$: this has the corresponding data $W = S_{n}$ for $n \geq 2$ and $\mathfrak{h} = \mathbb{C}^{n}$.  Note that there is only one conjugacy class of reflections, so $c \in \mathbb{R}$.  Irreducible representations $\tau$ of $S_{n}$ are labeled by partitions $\tau$ with weight $n$ and for each such representation, one can define the associated irreducible lowest weight representation $M_{c}(\tau)$ of $\mathbb{H}_{c}$.  Moreover, the representation $M_{c}(\tau)$ admits a unique (up to scaling) nondegenerate contravariant Hermitian form.  One may define a \textit{unitarity locus}: given $\lambda$ an irreducible representation of $W$, this is the set of parameters $U(\lambda)$ such that $M_{c}(\lambda)$ is unitary for $c \in U(\lambda)$.  Etingof, Stoica, and Griffeth determined the unitarity locus in type $A$ (as well as the dihedral and cyclic group cases).  There is a combinatorial picture for $M_{c}(\tau)$ as well: Suzuki provided an explicit combinatorial basis using \textit{periodic tableaux} of shape $\tau$.  There is also a natural weight function on the basis elements that may be described in terms of the associated periodic tableaux.  Thus, one may investigate the signature character in this context.

In Section 3, we construct a combinatorial algorithm that computes the sign of the norm of an arbitrary basis element.  The idea is analogous to the Hecke algebra case: we start with distinguished basis elements and give an explicit word of intertwiners that produces this arbitrary basis element, keeping track of how this word changes the norm.  This allows us to give an explicit formula for the signature character of $M_{c}(\tau)$:

\begin{theorem} \label{RCAinfsum}
Let $\tau$ be a partition such that $|\tau| = n$.  We have the formula for the signature character of $M_{c}(\tau)$:
\begin{equation*}
ch_{s}(M_{c}(\tau)) = \sum_{v \in \mathcal{B}(M_{c}(\tau))} t^{wt(v)} (-1)^{f(v)},
\end{equation*}
where $\mathcal{B}(M_{c}(\tau))$ is the explicit parametrization of a basis for $M_{c}(\tau)$ given in Definition \ref{basisparam} and formulas for $wt(\cdot)$ and $f( \cdot)$ are given in Definition \ref{functions} within the paper.
\end{theorem}

\noindent We then use the previous theorem to prove that the signature character of $M_{c}(\tau)$ is actually a rational function in $t$:

\begin{corollary} \label{rationalfcn}
Let $\tau$ be a partition such that $|\tau| = n$.  The signature character of $M_{c}(\tau)$ is of the following form:
\begin{equation*}
ch_{s}(M_{c}(\tau)) = \frac{p(t; c, \tau)}{(1-t)^{n}},
\end{equation*}
where $p(t; c, \tau)$ is a polynomial in $t$ depending on the data  $c$ and $\tau$.
\end{corollary}

\noindent The relationship between the signature of modules over $\mathcal{H}_{n}(q)$ and the signature character of modules over $\mathbb{H}_{c}$ is found in the \textit{asymptotic character}.  We define this to be 
\begin{equation} \label{asymchar}
a_{s}(M_{c}(\tau)):= ch_{s}(M_{c}(\tau)) (1-t)^{n}|_{t=1} = p(1;c,\tau).
\end{equation}
See also the remarks on page 16 within the paper for some motivation for this definition.  We will show that this provides the link between signatures of $\mathcal{H}_{n}(q)$ and signature characters of $\mathbb{H}_{c}$ in the following theorem:
\begin{theorem} \label{RCAHecke}
Let $\tau$ be a partition such that $|\tau| = n$ and $q = e^{2\pi i c}$.  We have
\begin{equation*}
a_{s}(M_{c}(\tau)) = s(S^{\tau}(q)).
\end{equation*}
\end{theorem}
\noindent We expect this to hold for other Coxeter groups, although this paper only investigates type $A$.

We mention the following connection to the KZ functor.  Recall in \cite{GGOR} it was shown that the KZ functor maps representations in category $\mathcal{O}$ of the rational Cherednik algebra $H_{c}(W)$ to representations of the Hecke algebra $\mathcal{H}_{q}(W)$, where $q = e^{2\pi i c}$.  As a consequence of the determination in \cite{ESG} of unitary representations in type $A$, it follows that in fact the KZ functor preserves unitarity: it maps unitary representations from $\mathcal{O}_{c}(S_{n}, \mathfrak{h})$ to unitary representations of $\mathcal{H}_{q}(S_{n})$ or zero (and the authors determined precisely which gets sent to zero).  

Finally, we investigate the asymptotic limit
\begin{equation*}
\lim_{c \rightarrow -\infty} ch_{s}(M_{c}(\tau)) = \lim_{c \rightarrow \infty} ch_{s}(M_{c}(\tau')),
\end{equation*}
(recall $\tau'$ is the conjugate of the partition $\tau$).  In this limiting case, we show that the formula of Theorem \ref{RCAinfsum} has a simpler form in terms of inversions and descents of permutations in $S(n)$.  For $\tau = (1^{n})$, the sign representation, we express the signature character in this limit in terms of explicit rational functions.  We provide several examples for small values of $n$.

We mention that W.-L. Yee studied an analogous problem in the context of Lie algebras, using the technology of wall crossings \cite{Yee1}.  She used a variant of Kazhdan-Lusztig polynomials, called signed Kazhdan-Lusztig polynomials, to compute the signature of an invariant Hermitian form on Verma modules and irreducible highest weight modules \cite{Yee2}.  Later, she found an explicit relationship between Kazhdan-Lusztig polynomials and her signed variant.  Her method can, in principle, be generalized to Cherednik and Hecke algebras, which would provide another approach to this problem that could work in all types.  

The outline of this paper is as follows.  In Section 2, we discuss some preliminaries pertaining to Hecke algebras and rational Cherednik algebras of type $A$.  In Section 3, we compute the signature of representations of Hecke algebras, and in Section 4 we provide some examples.  In Section 5, we turn to the rational Cherednik algebra case and compute the signature character of representations in this context.  We then relate the asymptotic signature to signatures of Hecke algebras.  In Section 6, we study the asymptotic limit $c \rightarrow -\infty$ of the signature character in the rational Cherednik algebra case.  Finally in Section 7, we provide some examples.

\medskip
\noindent\textbf{Acknowledgements.} The author would like to thank Pavel Etingof for suggesting this work, and for many helpful discussions and comments: in particular, for the explicit derivation of the formula for the signature character in the stable limit found in Theorem \ref{stablimit}.  She would also like to thank Monica Vazirani for many useful conversations and suggestions about this work.

\section{Preliminaries on Hecke algebras and rational Cherednik algebras}

Let $|q| = 1$.  Given an integer $m \neq 1$, let $[m]_{q} = \frac{q^{m}-1}{q-1} = 1+q+q^{2} + \cdots + q^{m-1}$ be the $q$-number.  Also let 
\begin{equation} \label{translation}
[[m]]_{q} = \frac{q^{m} - q^{-m}}{q-q^{-1}} = \frac{1}{q^{m-1}}[m]_{q^{2}} = \frac{\text{Im}(q^{m})}{\text{Im}(q)} \in \mathbb{R}.
\end{equation}  

 We use this to define
\begin{equation} \label{aisymbol}
a_{i} = a_{i}(q) = \begin{cases} 1,& [[i]]_{q} > 0 \\
z ,& [[i]]_{q} < 0;
\end{cases}
\end{equation}
note that it is a function of $q$ with values in the ring $\mathbb{Z}[z]/(z^{2}=1)$.  Note that $a_{1} = 1$ and $a_{i}^{2} = 1$ for any $i$.

Let $\{ \cdot \}: \mathbb{R} \setminus 0 \rightarrow \pm 1 $ denote the sign (positive or negative) of the real number within the brackets.

Recall the Hecke algebra $\mathcal{H}_{q} = \mathcal{H}_{n}(q)$ of $S_{n}$ with parameters $(1,-q)$ is the $\mathbb{C}$-algebra with generators $T_{1},...,T_{n-1}$ and relations $T_{i} T_{j} = T_{j}T_{i}$ if $|i - j| > 1$, $T_{i}T_{i+1}T_{i} = T_{i+1}T_{i}T_{i+1}$ and $(T_{i} +1)(T_{i} - q) = 0$.  We use the standard parametrization of $Irr(W )$ by partitions of $n$; $S^{\lambda} = S^{\lambda}(q)$ denotes the Specht module.  There is a Hermitian inner product $( \cdot, \cdot )$ on $S^{\lambda}$ as shown in \cite{St}; we also let $\sigma$ denote the involution on $\mathcal{H}_{q}$ from that paper.  With this notation, the signature of $S^{\lambda}$ may be expressed as
\begin{equation*}
s(S^{\lambda}) = \sum_{b \in \mathcal{B}} \{ ( b,b ) \},
\end{equation*}
where $\mathcal{B}$ is a basis for $S^{\lambda}$.  We note that it is the number of elements of positive norm minus the number with negative norm. 

Let $c \in \mathbb{R}$ be sufficiently generic.  We let $\mathbb{H}_{c}$ denote the rational Cherednik algebra of type $A$ (i.e., take $W = S_{n}$ for $n \geq 2$ and finite dimensional complex representation $\mathfrak{h}= \mathbb{C}^{n}$).  For $\lambda$ a representation of $S_{n}$, we have the Verma module $M_{c}(\lambda) := \mathbb{H}_{c} \otimes_{\mathbb{C}S_{n} \ltimes S\mathfrak{h}} \lambda$ (it is the induced module from $S_{n} \ltimes S\mathfrak{h}$).  Then for $|\lambda| = n$ and $\lambda$ a partition, $M_{c}(\lambda)$ denotes the unique irreducible quotient of the Verma module $M_{c}(\lambda)$.  There is a natural grading on $M_{c}(\tau)$, which we will denote by $w$.  As in \cite{ESG}, there is a unique (up to scaling) contravariant Hermitian form on $M_{c}(\tau)$, denoted by $\langle \cdot, \cdot \rangle$.  With this notation, the signature character may be expressed as
\begin{equation*}
ch_{s}(M_{c}(\tau)) = \sum_{b \in \mathcal{B}} t^{w(b)} \{ \langle b,b \rangle \},
\end{equation*}  
where $\mathcal{B}$ is a basis for $M_{c}(\tau)$.  We note that for a fixed weight $m$, the coefficient on $t^{m}$ is the number of elements in that weight space of positive norm, minus the number with negative norm.

\section{Signatures of representations of Hecke algebras} 

We recall some notation involving compositions and tableaux; we follow \cite{M}.  

We say $\mu = (\mu_{1}, \mu_{2}, \dots)$, for $\mu_{i} \in \mathbb{Z}_{\geq 0}$, is a composition of $m \in \mathbb{Z}_{>0}$ if  $\sum_{i} \mu_{i} = m$.  The $\mu_{i}$ are the parts of $\mu$ and $|\mu| = \sum_{i} \mu_{i}$ is the weight.  If additionally the parts of $\mu$ are in non-increasing order, we say that $\mu$ is a partition.

The diagram of a composition $\mu$ is the subset 
\begin{equation*}
[\mu] = \{ (i,j) | 1 \leq j \leq \mu_{i} \text{ and } i \geq 1 \}
\end{equation*}
of $\mathbb{N} \times \mathbb{N}$.  The elements of $[\mu]$ are called the nodes of $\mu$.  If $\mu$ is a composition of $n$ then a $\mu$-tableau is a bijection $t: [\mu] \rightarrow \{1,2, \dots, n\}$ and we write $\text{Shape}(t) = \mu$.  If $\mu$ is a partition, we will write $\text{Std}(\mu)$ for the set of standard $\mu$-tableaux: entries increase from left to right in each row and from top to bottom in each column in the corresponding diagram.

Let $x = (i,j)$ be a node in $[\lambda]$.  The $e$-residue of $x$ is the integer $\text{res}(x) = j-i \text{ mod } e$.  If $T \in \text{Std}(\lambda)$ and $k$ is an integer with $1 \leq k \leq n$ then the $e$-residue of $k$ in $T$ is $\text{res}_{T}(k) = \text{res}(x)$, where $x$ is the unique node in $[\lambda]$ such that $T(x) = k$.  The content vector of the tableaux $T$ is the string of $e$-residues $(\text{res}_{T}(1), \text{res}_{T}(2), \dots, \text{res}_{T}(n))$, where $e \rightarrow \infty$.

Let $T \downarrow m$ be the tableaux obtained by deleting all cells containing elements $\geq m+1$ from $T$.  

Recall the following partial order on compositions.  Let $\mu$ and $\nu$ be compositions.  Then $\nu \preceq \mu$ if 
\begin{equation*}
\sum_{j=1}^{i} \nu_{j} \leq \sum_{j=1}^{i} \mu_{j}
\end{equation*}
for all $i \geq 1$.  We can use this to define a partial order on standard tableaux as follows.  Let $s,t$ be standard tableaux.  Then $t \trianglelefteq s$ if $\text{Shape}(t \downarrow m) \preceq \text{Shape}(s \downarrow m)$ for all $m \geq 1$.

Recall that $S^{\lambda}$ has a basis indexed by standard tableaux $t$ of shape $\lambda$, we will denote the corresponding basis element by $f_{t}$.  Let $t^{\lambda}$ denote the tableau with filling $1,2, \cdots, n$ from left to right in order in rows from top to bottom.  Let $(i,i+1)$ denote the operation on tableaux that swaps $i$ and $i+1$.

In this section, we will prove Theorem \ref{Heckesig}; we will first prove the following two propositions.

\begin{proposition} \cite[Theorem 3.34]{M} \label{DJ}
Let $s,t \in \text{Std}(\lambda)$ and $t = s(i,i+1)$.  Then we have
\begin{equation*}
f_{s}T_{i} = \begin{cases} \frac{-1}{[\rho]_{q}}f_{s} + f_{t} ,& s \triangleright t \\
\frac{q[\rho+1]_{q}[\rho-1]_{q}}{[\rho]_{q}^{2}}f_{t} + \frac{q^{\rho}}{[\rho]_{q}}f_{s} ,& t \triangleright s.
\end{cases}
\end{equation*}
\end{proposition}

\begin{remarks}
This differs from \cite[Theorem 3.34]{M} in the factor of $q$ in the second case above on the coefficient of $f_{t}$.  We provide part of the proof below to illustrate this discrepancy.
\end{remarks}

\begin{proof}
We have, from \cite[Theorem 3.34]{M}, for $t = s(i,i+1)$ with $t \in \text{Std}(\lambda)$ and $s \triangleright t$,
\begin{equation}\label{caseone}
f_{s}T_{i} = \frac{-1}{[\rho]_{q}}f_{s} + f_{t}.
\end{equation}
Applying $T_{i}$ to both sides gives
\begin{equation*}
f_{s}T_{i}^{2} = \frac{-1}{[\rho]_{q}}f_{s}T_{i} + f_{t}T_{i}.
\end{equation*}
Using the quadratic equation $T_{i}^{2} = q + (q-1)T_{i}$, we rewrite this as
\begin{equation*}
qf_{s} + (q-1)f_{s}T_{i} = \frac{-1}{[\rho]_{q}}f_{s}T_{i} + f_{t}T_{i}.
\end{equation*}
Using (\ref{caseone}), we get
\begin{equation*}
qf_{s} + (q-1)\Big(\frac{-1}{[\rho]_{q}}f_{s} + f_{t}\Big) = \frac{-1}{[\rho]_{q}}\Big(\frac{-1}{[\rho]_{q}}f_{s} + f_{t}\Big) + f_{t}T_{i}.
\end{equation*}
Thus,
\begin{equation*}
f_{t}T_{i} = \Big(q - \frac{1}{[\rho]_{q}^{2}} - \frac{q-1}{[\rho]_{q}} \Big)f_{s} + \Big((q-1) + \frac{1}{[\rho]_{q}} \Big) f_{t}.
\end{equation*}
Finally,
\begin{equation*}
(q-1) + \frac{1}{[\rho]_{q}} = \frac{(q-1)[\rho]_{q} + 1}{[\rho]_{q}} = \frac{q^{\rho}}{[\rho]_{q}}
\end{equation*}
and
\begin{multline*} 
q + \frac{1-q}{[\rho]_{q}} - \frac{1}{[\rho]_{q}^{2}} = \frac{q[\rho]_{q}^{2} + (1-q)[\rho]_{q} - 1}{[\rho]_{q}^{2}} = \frac{q[\rho]_{q}^{2} - q^{\rho}}{[\rho]_{q}^{2}} \\ = \frac{q(1-q^{\rho})^{2} - q^{\rho}(1-q)^{2}}{(1-q^{\rho})^{2}} 
= \frac{q(1-2q^{\rho}+q^{2\rho}) - q^{\rho}(1-2q+q^{2})}{(1-q^{\rho})^{2}} \\= \frac{q + q^{2\rho+1} - q^{\rho} - q^{\rho+2}}{(1-q^{\rho})^{2}}  = \frac{q(1-q^{\rho+1}) - q^{\rho}(1-q^{\rho+1})}{(1-q^{\rho})^{2}} = \frac{q[\rho+1]_{q}[\rho-1]_{q}}{[\rho]_{q}^{2}},
\end{multline*}
as desired.

Thus, for $t = s(i,i+1)$ with $t \in \text{Std}(\lambda)$ and $t \triangleright s$, we have
\begin{equation*}
f_{s}T_{i} = \frac{q[\rho+1]_{q}[\rho-1]_{q}}{[\rho]_{q}^{2}}f_{t} + \frac{q^{\rho}}{[\rho]_{q}}f_{s}.
\end{equation*}

\end{proof}

We now use the previous result to compute the relationship between the norm of $f_{t}$ and the norm of $f_{s}$, where $t = s(i,i+1)$.

\begin{proposition} \label{DJnorm}
Let $s$ be a standard $\lambda$-tableaux, $t = s(i,i+1)$ and $\rho = \text{res}_{s}(i) - \text{res}_{t}(i)$.  Then
\begin{equation*}
(f_{t}, f_{t}) = \begin{cases} \frac{[\rho-1]_{q}[\rho+1]_{q}}{[\rho]_{q}^{2}}(f_{s}, f_{s}), & \text{if } t \in \text{Std}(\lambda) \text{ and }s \triangleright t \\
\frac{[\rho]_{q}^{2}}{[\rho+1]_{q}[\rho-1]_{q}}( f_{s},f_{s} ), & \text{if } t \in \text{Std}(\lambda) \text{ and } t \triangleright s.
\end{cases}
\end{equation*}
\end{proposition}
\begin{proof}
We have
\begin{equation*}
( f_{s}T_{i}, f_{s}T_{i} ) = (f_{s}, f_{s}T_{i}\sigma(T_{i}^*)) = (f_{s}, f_{s}T_{i}T_{i}^{-1}) = (f_{s}, f_{s}).
\end{equation*}
Using Proposition \ref{DJ} on the left hand side above, we obtain the equality
\begin{equation*}
\frac{1}{[\rho]_{q}\overline{[\rho]_{q}}} ( f_{s},f_{s} ) + ( f_{t}, f_{t} ) = (f_{s},f_{s})
\end{equation*}
in the case $t \in \text{Std}(\lambda) \text{ and }s \triangleright t$.  Thus,
\begin{equation*}
( f_{t}, f_{t} )= \Big(1 - \frac{1}{[\rho]_{q}\overline{[\rho]_{q}}} \Big)( f_{s}, f_{s} )
\end{equation*}
in this case.

Finally, we compute
\begin{multline*} 
1 - \frac{1}{[\rho]_{q}\overline{[\rho]_{q}}} = \frac{(1-q^{\rho})(1-\bar{q}^{\rho})}{(1-q^{\rho})(1-\bar{q}^{\rho})} - \frac{(1-q)(1-\bar{q})}{(1-q^{\rho})(1-\bar{q}^{\rho})} = \frac{-q^{\rho} - \bar{q}^{\rho} + q + \bar{q}}{(1-q^{\rho})(1-\bar{q}^{\rho})} \\
= \frac{-q^{\rho} - q^{-\rho} + q + q^{-1}}{(1-q^{\rho})(1-q^{-\rho})} = \frac{q^{2\rho} + 1 - q^{\rho+1} - q^{\rho-1}}{(1-q^{\rho})^{2}} = \frac{(1-q^{\rho-1})(1-q^{\rho+1})}{(1-q^{\rho})^{2}} \\= \frac{[\rho-1][\rho+1]}{[\rho]^{2}}
\end{multline*}
as desired.

For the case $t \in \text{Std}(\lambda)$, $t \triangleright s$, we note that $t(i,i+1) = s$ so using the previous computation
\begin{equation*}
( f_{s}, f_{s} ) = \frac{[\rho-1]_{q}[\rho+1]_{q}}{[\rho]_{q}^{2}}( f_{t},f_{t} ).
\end{equation*}
Thus,
\begin{equation*}
( f_{t}, f_{t} ) = \frac{[\rho]_{q}^{2}}{[\rho-1]_{q}[\rho+1]_{q}} ( f_{s}, f_{s} )
\end{equation*}
as desired.
\end{proof}

\begin{remarks}
Note that, under the transformation of parameters $q \rightarrow q^{2}$, Proposition \ref{DJnorm} states (for example in the first case)
\begin{multline*}
(f_{t}, f_{t}) = \frac{[\rho-1]_{q^{2}}[\rho+1]_{q^{2}}}{[\rho]_{q^{2}}^{2}} (f_{s},f_{s}) = \frac{[[\rho-1]]_{q}q^{\rho-2}[[\rho+1]]_{q}q^{\rho}}{[[\rho]]_{q}^{2}q^{2(\rho-1)}} (f_{s},f_{s}) \\ = \frac{[[\rho-1]]_{q}[[\rho+1]]_{q}}{[[\rho]]_{q}^{2}} (f_{s},f_{s}),
\end{multline*}
by virtue of Equation (\ref{translation}).  Thus, the norm change factor is real-valued.  From now on, we will use the transformation $q \rightarrow q^{2}$.
\end{remarks}

We will now use the previous propositions to prove Theorem \ref{Heckesig}.  In fact, using the symbols in Equation (\ref{aisymbol}), we will prove something slightly stronger.  Namely, we will compute a modification of the signature, $s_{z}(S^{\lambda})$, which is an element of $\mathbb{Z}[z]/(z^{2}-1)$.  We have
\begin{equation*}
s_{-1}(S^{\lambda}) = s(S^{\lambda})
\end{equation*}
and
\begin{equation*}
s_{1}(S^{\lambda}) = \dim S^{\lambda};
\end{equation*}
moreover $s_{-1} + s_{1}$ is twice the number of elements with positive norm.

\begin{theorem} \label{zHeckesig}
Let $\lambda$ be a partition of $n$.  Then
\begin{equation*}
s_{z}(S^{\lambda}) = \sum_{\substack{T \in Std(\lambda) \\ (d_{1}, \dots, d_{n}) \text{ content vector of } T}} \prod_{\substack{1 \leq l < i \leq n \\ d_{i}-d_{l} < 0}} a_{d_{i}-d_{l}+1} a_{d_{i}-d_{l}-1}.
\end{equation*}
\end{theorem}

\begin{proof}
Let $T$ be a tableaux of shape $\lambda$ with content vector $(d_{1}, \dots, d_{n})$.  Suppose switching $i$ and $i+1$ in tableaux $T$ results in tableaux $T'$.  Note that $T'$ then has content vector $$(d_{1}, \dots, d_{i-1}, d_{i+1}, d_{i}, d_{i+2}, \dots, d_{n}).$$  We will write $N(T; \lambda)$ for the signature of the norm of $f_{T}$.  We will show that the norm change between $N(T; \lambda)$ and $N(T;\lambda')$ in the formula above agrees with that of Proposition \ref{DJnorm}.  One can deduce that in $T$ the set of positions of $i$ and $i+1$ are of the form $\{(a,b), (c,d)\}$ with $c<a$ and $d>b$, i.e., $b-a < d-c$.  Note that $N(T; \lambda)$ and $N(T'; \lambda)$ only differ in the contribution resulting from $(d_{i}, d_{i+1})$.  We distinguish between two cases:

\textbf{Case 1:} In $T$, $i$ is in position $(a,b)$ and $i+1$ is in position $(c,d)$, then
\begin{multline*}
N(T'; \lambda) = N(T;\lambda) [[d_{i}-d_{i+1}+1]]_{q} [[d_{i}-d_{i+1}-1]]_{q} \\= N(T;\lambda) [[(b-a) - (d-c)+1]]_{q} [[(b-a) - (d-c)-1]]_{q}.
\end{multline*}

\textbf{Case 2:} In $T$, $i$ is in position $(c,d)$ and $i+1$ is in position $(a,b)$, then
\begin{multline*}
N(T'; \lambda) = \frac{N(T;\lambda)}{[[d_{i+1}-d_{i}+1]]_{q} [[d_{i+1}-d_{i}-1]]_{q}} \\= \frac{N(T;\lambda)}{[[(b-a)-(d-c)+1]]_{q} [[(b-a)-(d-c)-1]]_{q}} \\ = \frac{N(T;\lambda)}{[[(d-c)-(b-a)+1]]_{q} [[(d-c)-(b-a)-1]]_{q}},
\end{multline*}
since for any $N$ we have
\begin{multline*}
[[N+1]]_{q} [[N-1]]_{q} = \frac{(q^{N+1}-q^{-N-1})(q^{N-1}-q^{-N+1})}{(q-q^{-1})^{2}} \\ = \frac{(q^{-N-1}-q^{N+1})(q^{-N+1}-q^{N-1})}{(q-q^{-1})^{2}} = [[-N+1]]_{q} [[-N-1]]_{q}.
\end{multline*}

Note that the conditions of Case 1 are equivalent to $T \trianglerighteq T'$, and Case 2 is equivalent to $T' \trianglerighteq T$.  So this agrees with Proposition \ref{DJnorm} since $[[\rho]]^{2}$ has sign 1 (where $\rho = res_{T}(i) - res_{T'}(i)) = (b-a) - (d-c), \text{ or } (d-c) - (b-a)$ as above).  
\end{proof}

\begin{definition}
Let $T \in Std(\lambda)$.  Let 
\begin{equation*}
Config(T) = \{(l,i): 1 \leq l<i \leq n \text{ and } l \text{ is located above and to the right of }i \text{ in } T\}.
\end{equation*}
\end{definition}

\begin{corollary}
Let $\lambda$ be fixed, with $|\lambda| = n$.  We have 
\begin{equation*}
s_{z}(S^{\lambda}) = \sum_{\substack{T \in Std(\lambda) \\ (d_{1}, \dots, d_{n}) \text{ content vector of }T}} \prod_{(l,i) \in Config(T)} a_{d_{l} - d_{i} + 1}(q) a_{d_{l}-d_{i}-1}(q).
\end{equation*}
\end{corollary}

\begin{proof}
From Theorem \ref{zHeckesig}, we have 
\begin{equation*}
s_{z}(S^{\lambda}) = \sum_{\substack{T \in Std(\lambda) \\ (d_{1}, \dots, d_{n}) \text{ content vector of } T}} \prod_{1 \leq i \leq n} \prod_{\substack{1 \leq l \leq i-1 \\ d_{i} - d_{l} < 0}} a_{d_{i}-d_{l} + 1} a_{d_{i}-d_{l} - 1}.
\end{equation*}
Let $T \in Std(\lambda)$ be fixed with $1 \leq i \leq n$ be in position $(x,y)$ of $T$.  Then we note that any $l<i$ that satisfies $d_{i} - d_{l} < 0$ must be in position $(x',y')$ for $x'<x$.  Also note that any $m$ in position $(x', y')$ for $x'<x$ and $y' \leq y$ satisfies $m<i$.  Thus, we may write the above sum as
\begin{multline*}
\prod_{(x,y) \in Sh(\lambda)} \prod_{\substack{(x',y') \in Sh(\lambda) \\ x'<x, y'<y \\ y-x<y'-x'}} a_{(y'-x') - (y-x)+1}(q) a_{(y'-x')-(y-x)-1}(q) \times \\\sum_{\substack{T \in Std(\lambda) \\ }} \prod_{\substack{(x,y) \in Sh(\lambda) \\ i \text{ in position }(x,y)}} \prod_{\substack{1 \leq l<i \\ l \text{ in position } (x',y') \\ \text{ with } x'<x, y'>y}} a_{(y'-x') - (y-x)+1}(q) a_{(y'-x')-(y-x)-1}(q).
\end{multline*}
Rephrasing this using content vectors and the set $Config(T)$ gives the result.
\end{proof}

\section{Examples}

In this section, we will use Theorem \ref{zHeckesig} to compute signatures for some special shapes and for some small values of $n$.

\medskip
1. Reflection representation $\lambda = (n-1, 1)$. 
\begin{proposition}
Let $\lambda = (n-1, 1)$.  Then we have
\begin{equation*}
s_{z}(S^{\lambda}) =  a_{1}a_{2} + a_{2}a_{3} + \cdots + a_{n-1}a_{n}.
\end{equation*}
\end{proposition}

\begin{proof}
We label basis elements $T_{i}$ for $2 \leq i \leq n$ where cell $(2,1)$ contains $i$.  The signature of $T_{i}$, using Theorem \ref{Heckesig} is 
\begin{equation*}
a_{0}a_{2}a_{1}a_{3}a_{2}a_{4} \cdots a_{i-2}a_{i} = a_{0}a_{1}a_{i-1}a_{i},
\end{equation*}
since we get contributions from $(l, i)$ for $1 \leq l < i$.  We have also used that, for $N < 0$, $a_{N+1}a_{N-1} = a_{-N+1}a_{-N-1}$.

So we have, summing over all elements $T_{i}$ for $2 \leq i \leq n$ (and pulling out a common multiple factor)
\begin{equation*}
s_{z}(S^{\lambda}) = a_{1}a_{2} + a_{2}a_{3} + \cdots + a_{n-1}a_{n}.
\end{equation*}

We note that this implies that $S^{\lambda}$ is unitary if and only if $a_{i-1}a_{i} = 1$ for $2 \leq i \leq n$.  So $a_{i-1} = a_{i} = 1$ or $a_{i-1} = a_{i} = z$ for $2 \leq i \leq n$.  But since $a_{1} = 1$, this only happens if $a_{i} = 1$ for $2 \leq i \leq n$.  
\end{proof}
\medskip
\medskip
\medskip

2.  Hook representation $\lambda = (n-l, 1^{l})$.  

\begin{proposition}
Let $\lambda = (n-l, 1^{l})$.  Then we have
\begin{equation*}
s_{z}(S^{\lambda}) 
= \sum_{2 \leq j_{1} < j_{2}< \cdots < j_{l} \leq n} \prod_{i=1}^{l} a_{j_{i}-1} a_{j_{i}}.
\end{equation*}
\end{proposition}

\begin{proof}
We label basis elements $T_{(j_{1}, j_{2}, \dots, j_{l})}$ for $2 \leq j_{1} < j_{2}< \cdots < j_{l} \leq n$.  As in the previous example, one can use Theorem \ref{Heckesig} to show that the signature of $T_{(j_{1}, j_{2}, \dots, j_{l})}$ is
\begin{equation*}
\prod_{i=1}^{l} a_{j_{i}-1}a_{j_{i}}.
\end{equation*}
\end{proof}

In particular, the previous example (reflection representation) is $l=1$ here.

\medskip
\medskip
\medskip

3. Two row $\lambda = (n-m, m)$ with $0 \leq m \leq \lfloor n/2 \rfloor$.  

\begin{proposition}
Let $\lambda = (n-m, m)$ with $0 \leq m \leq \lfloor n/2 \rfloor$.  Then we have
\begin{equation*}
s_{z}(S^{\lambda}) =\sum_{\substack{1 \leq j_{1} < \dots < j_{m} \leq n \\ j_{i} \geq 2i}} \prod_{i=1}^{m} a_{j_{i} - (2i-1)}a_{j_{i}-2(i-1)}.
\end{equation*}
\end{proposition}

\begin{proof}
We label elements $T_{(j_{1}, \dots, j_{m})}$, with $1 \leq j_{1} < \dots < j_{m} \leq n$ in the second row with the condition that $j_{i} \geq 2i$ for all $i$.  As in the previous examples, we compute the norm of $T_{(j_{1}, \dots, j_{m})}$ using Proposition \ref{Heckesig}, and sum over all basis elements.  

\end{proof}

In particular, $m=1$ is the reflection representation case above.

\medskip
\medskip
\medskip

4. The case $n=3$.  The possible partitions are $\lambda = (3), (2,1), (1,1,1)$.

We have, using the reflection representation formula,
\begin{equation*}
s_{z}(S^{(2,1)}) = \frac{a_{3}}{a_{2}}\Big(a_{1}a_{2} + a_{2}a_{3}\Big) = a_{1}a_{3} + 1 .  
\end{equation*}

\medskip
\medskip
\medskip

5.  The case $n=4$.  The possible partitions are $\lambda = (4), (1,1,1,1), (3,1), (2,2)$ and $(2,1,1)$.

We have, using the reflection representation formula,
\begin{equation*}
s_{z}(S^{(3,1)}) = \frac{a_{4}}{a_{3}}\Big(a_{1}a_{2} + a_{2}a_{3} + a_{3}a_{4}\Big) = a_{1}a_{2}a_{3}a_{4} + a_{2}a_{4} + 1
\end{equation*}
and, using the hook formula,
\begin{multline*}
s_{z}(S^{(2,1,1)}) = \frac{a_{4}}{a_{2}}\Big(a_{2}a_{3}a_{1}a_{2} + a_{2}a_{4}a_{1}a_{3} + a_{3}a_{4}a_{2}a_{3}\Big) \\
=\frac{a_{4}}{a_{2}} \Big( a_{1}a_{3} + a_{1}a_{2}a_{3}a_{4} + a_{2}a_{4}\Big) = a_{1}a_{2}a_{3}a_{4} + a_{1}a_{3} + 1.
\end{multline*}

Also, using the two-row formula,
\begin{equation*}
s_{z}(S^{(2,2)}) = \frac{a_{3}}{a_{1}} \Big( a_{2- 1}a_{2}a_{4-3}a_{4-2} + a_{3-1}a_{3}a_{4-3}a_{4-2} \Big) 
= \frac{a_{3}}{a_{1}} \Big(1 + a_{3} \Big) = a_{3} + 1.
\end{equation*}  

\medskip
\medskip
\medskip

6.  The case $n=5$.  The possible partitions are $\lambda = (5), (1,1,1,1,1), (3,2)$ and \\ $(3,1,1), (4,1), (2,2,1), (2,1,1,1)$.

We have, using the reflection representation formula,
\begin{equation*}
s_{z}(S^{(4,1)}) = \frac{a_{5}}{a_{4}}\Big( a_{1}a_{2} + a_{2}a_{3} + a_{3}a_{4} + a_{4}a_{5}\Big) = a_{1}a_{2}a_{4}a_{5} + a_{2}a_{3}a_{4}a_{5} + a_{3}a_{5} + 1
\end{equation*}
and, using the hook formula,
\begin{multline*}
s_{z}(S^{(3,1,1)}) = \frac{a_{5}}{a_{3}} \Big( a_{2}a_{3}a_{1}a_{2} + a_{2}a_{4}a_{1}a_{3} + a_{2}a_{5}a_{1}a_{4} + a_{3}a_{4}a_{2}a_{3} + a_{3}a_{5}a_{2}a_{4} + a_{4}a_{5}a_{3}a_{4}\Big) \\
= \frac{a_{5}}{a_{3}} \Big( a_{1}a_{3} + a_{1}a_{2}a_{3}a_{4} + a_{1}a_{2}a_{4}a_{5} + a_{2}a_{4} + a_{2}a_{3}a_{4}a_{5} +   a_{3}a_{5}  \Big) \\
= a_{1}a_{5} + a_{1}a_{2}a_{4}a_{5} + a_{1}a_{2}a_{3}a_{4} + a_{2}a_{3}a_{4}a_{5} + a_{2}a_{4} + 1.
\end{multline*}

Also, using the hook formula,
\begin{multline*}
s_{z}(S^{(2,1,1,1)}) = \frac{a_{5}}{a_{2}} \Big( a_{2}a_{3}a_{4}a_{1}a_{2}a_{3} + a_{2}a_{3}a_{5}a_{1}a_{2}a_{4}  + a_{2}a_{4}a_{5}a_{1}a_{3}a_{4} + a_{3}a_{4}a_{5}a_{2}a_{3}a_{4} \Big) \\
= \frac{a_{5}}{a_{2}}\Big( a_{1}a_{4} + a_{1}a_{3}a_{4}a_{5} + a_{1}a_{2}a_{3}a_{5} + a_{2}a_{5} \Big) = a_{1}a_{2}a_{4}a_{5} + a_{1}a_{2}a_{3}a_{4} + a_{1}a_{3} + 1.
\end{multline*}

We have, using the two-row formula,
\begin{multline*}
s_{z}(S^{(3,2)}) = \frac{a_{4}}{a_{2}}\Big(a_{1}a_{2}a_{1}a_{2} + a_{1}a_{2}a_{2}a_{3} + a_{2}a_{3}a_{1}a_{2}
+ a_{2}a_{3}a_{2}a_{3} + a_{3}a_{4}a_{2}a_{3} \Big) \\
= \frac{a_{4}}{a_{2}}\Big( 1 + a_{1}a_{3} + a_{3} + 1 + a_{2}a_{4} \Big) = a_{2}a_{4} + a_{1}a_{2}a_{3}a_{4} + a_{2}a_{3}a_{4} + a_{2}a_{4} + 1.
\end{multline*}

Finally we compute $\lambda = (2,2,1)$.  Writing the tableaux according to their rows, we have $t^{\lambda} = (1,2;3,4;5)$, $t_{1} = (1,2;3,5;4)$, $t_{2} = (1,3;2,4;5)$, $t_{3} = (1,3;2,5;4)$ and $t_{4} = (1,4;2,5;3)$.  We use Theorem \ref{Heckesig} to compute

\begin{equation*}
s_{z}(S^{(2,2,1)}) = 2 + 2a_{3}  + a_{2}a_{4}.
\end{equation*}

\section{Signature characters of representations of rational Cherednik algebras}

Let $c$ be sufficiently generic (as mentioned in the Introduction, we take $c$ in the intervals set up by avoiding possible points of degeneracy), $\kappa = -1/c$ and $\lambda$ a partition with $|\lambda| = n$.  We recall the definition of $M_{c}(\lambda)$ in Section 2.  By \cite{S}, a basis for the module $M_{c}(\lambda)$ is given by periodic tableaux of shape $\lambda$.  We recall the definition here.

\begin{definition}
Let $\lambda$ be visualized as a subset of $\mathbb{Z} \times \mathbb{Q}$ with points $(i,j)$ for $i,j \in \mathbb{Z}$ for $1 \leq i \leq m$ and $1 \leq j \leq \lambda_{i}$.  Let $p = (-m, \kappa - m)$ and $\hat{\lambda} = \lambda + \mathbb{Z}p \subset \mathbb{Z} \times \mathbb{Q}$.  A \textbf{periodic tableaux} on $\hat{\lambda}$ is a bijection $T: \hat{\lambda} \rightarrow \mathbb{Z}$ such that 
\begin{enumerate}
\item for all $b \in \hat{\lambda}$, $T(b+p) = T(b)-n$
\item $(a,b),(a,b+1) \in \hat{\lambda}$ implies $T(a,b) < T(a,b+1)$
\item $(a,b), (a+k+1,b+k) \in \hat{\lambda}$ for $k \in \mathbb{Z}_{\geq 0}$ implies $T(a,b) < T(a+k+1, b+k)$.
\end{enumerate}
\end{definition}

\begin{definition}
Let $T$ be a periodic tableaux.  The \textbf{content vector} is
\begin{equation*}
ct(T) = (ct(T^{-1}(1)), \dots, ct(T^{-1}(n))),
\end{equation*}
where $ct(a,b) = b-a$.
\end{definition}

We will also write permutations $\sigma \in S_{n}$ in one-line notation.

We note that a periodic tableaux is entirely determined by the central block (i.e., the map $T$ on $\lambda$ above, which is a subset of $\hat{\lambda}$), since the rest is determined by shifting by $n$.  The central block consists of $n$ positive integers, where every residue class modulo $n$ is present.  

We will also need to introduce certain intertwining operators, see \cite[Appendix]{ESG} for details.  Let 
\begin{equation*}
\sigma_{i} = s_{i} - \frac{1}{z_{i} - z_{i+1}}
\end{equation*}
\begin{equation*}
\Phi = x_{n}s_{n-1} \cdots s_{1}
\end{equation*}
\begin{equation*}
\Psi = y_{1}s_{1} \cdots s_{n-1}.
\end{equation*}
These map eigenvectors for $z_{1}, \dots, z_{n}$ to eigenvectors and satisfy
\begin{equation*}
\sigma_{i}^{2} = \frac{(z_{i}-z_{i+1})^{2}-1}{(z_{i}-z_{i+1})^{2}}
\end{equation*}
and
\begin{equation*}
\Psi\Phi = z_{1}.
\end{equation*}
Finally $\sigma_{i}^{*} = \sigma_{i}$ and $\Phi^{*} = \Psi$.

\begin{lemma}
 Let $p$ be a periodic tableaux.  Then its central block (1) consists of entries $\{g_{1}n + \alpha_{1}, g_{2}n + \alpha_{2}, \dots, g_{n}n+\alpha_{n} \}$ for $g_{i} \in \mathbb{Z}_{+}$ and $\{\alpha_{1}, \alpha_{2}, \dots, \alpha_{n} \} = \{1,2, \dots, n \}$, and (2) is order isomorphic to some standard tableaux $t$. 
\end{lemma}

\begin{proof}
(1) follows from the fact that $p$ must contain all of $\mathbb{Z}$, and $p$ is determined from the central block by shifting by $n$.  It is also a restriction that the central block must contain positive entries. (2) follows since $p$ must be strictly increasing from left to right along rows, and from top to bottom along columns.  
\end{proof}

\begin{definition} 
For a fixed $t \in \text{Std}(\lambda)$, let the map $\Psi_{t}: \mathbb{Z}^{n} \rightarrow \text{Std}(\lambda)$ satisfy
\begin{equation*}
\Psi_{t}(a_{1}, \dots, a_{n}) = t',
\end{equation*}
for $a_{1} < a_{2} < \dots < a_{n}$ and $t'$ is the tableaux obtained by replacing $1$ by $a_{1}$, $2$ by $a_{2}$, through $n$ by $a_{n}$ in the tableaux $t$.  (For strings not in increasing order, the map is zero.) Thus, the resulting tableaux $t'$ is order isomorphic to $t$ but does not only contain integers in the set $\{1,2, \dots, n \}$. 
\end{definition}

By a slight abuse of notation, we may write $(a_{1}, \dots, a_{n})$ instead of $\Psi_{t}(a_{1}, \dots, a_{n})$ when it is clear from context.

\begin{lemma}
Let $t$ be a standard tableaux and fix $g_{1} \leq g_{2} \dots \leq g_{n} \in \mathbb{Z}_{+}$.  Then the central blocks containing entries $g_{1}n+\alpha_{1}, \dots, g_{n}n+\alpha_{n}$ (with the only restriction on $\alpha_{i}$ being $\{\alpha_{1}, \dots, \alpha_{n} \} = \{1, 2, \dots, n \}$) which are order isomorphic to $t$ are 
\begin{equation*}
(g_{1}n + \alpha_{1}, g_{2}n + \alpha_{2}, \dots, g_{n}n + \alpha_{n}),
\end{equation*}
where if $i<j$ and $g_{i} = g_{j}$, then $\alpha_{i}<\alpha_{j}$.  In particular, writing $\mu = (g_{1}, \dots, g_{n})$, there are 
\begin{equation*}
\frac{n!}{m_{1}(\mu)! m_{2}(\mu)! \cdots}
\end{equation*}
many such central blocks.
\end{lemma}
\begin{proof}
Obvious.
\end{proof}

\begin{definition} \label{basisparam}
Let $\mathcal{B}(M_{c}(\lambda))$ be the following parametrization of a basis for $M_{c}(\lambda)$: the data
\begin{itemize}
\item $g_{1} \leq g_{2} \leq \cdots \leq g_{n} \in \mathbb{Z}_{+}$
\item $\alpha \in S_{n}$ such that if $g_{i} = g_{j}$ for $i<j$ then $\alpha(i) < \alpha(j)$
\item $t \in Std(\lambda)$
\end{itemize}
produces a periodic tableaux of shape $\lambda$ with $(g_{1}n + \alpha(1), g_{2}n + \alpha(2), \dots, g_{n}n + \alpha(n))$ in the central block (under the map $\Psi_{t}$).
\end{definition}

\begin{lemma}
Fix $t \in Std(\lambda)$ and suppose it has content vector $(d_{1}, \dots, d_{n})$.  Also fix $g_{1} \leq g_{2} \leq \cdots \leq g_{n}$.  Then $\Psi_{t}(g_{1}n + 1, g_{2}n + 2, \dots, g_{n}n+n)$ has content vector 
\begin{equation*}
(d_{1} - \kappa g_{1}, d_{2} - \kappa g_{2}, \dots, d_{n} - \kappa g_{n}).  
\end{equation*}
\end{lemma}

\begin{proof}
Suppose $i$, for $1 \leq i \leq n$ is in position $(a,b)$ of $t$, so in particular $d_{i} = b-a$.  Then by the shifting property, $i$ is in position $(a,b) - g_{i}(-m, \kappa-m)$ of $\Psi_{t}(g_{1}n + 1, g_{2}n + 2, \dots, g_{n}n+n)$, so the $i^{th}$ entry of the content vector is $d_{i} -g_{i}\kappa$, as desired.
\end{proof}

\begin{lemma}
Fix $t \in Std(\lambda)$, and $g_{1} \leq g_{2} \leq \cdots \leq g_{n}$.  Suppose $\Psi_{t}(g_{1}n + 1, g_{2}n + 2, \dots, g_{n}n + n)$ has content vector $(c_{1}, \dots, c_{n})$.  Let $\sigma \in S_{n}$, then $\Psi_{t}(g_{1}n + \sigma(1), g_{2}n + \sigma(2), \dots, g_{n}n + \sigma(n))$ has content vector $(c_{\sigma^{-1}(1)}, c_{\sigma^{-1}(2)}, \dots, c_{\sigma^{-1}(n)})$.  
\end{lemma}

\begin{proof}
Let $1 \leq i \leq n$.  Then $i = \sigma(j)$ for some $j$, and the $j^{th}$ entry of the content vector of $\Psi_{t}(g_{1}n + 1, g_{2}n + 2, \dots, g_{n}n + n)$ is $c_{j}$.  Thus, the $i^{th}$ entry of the content vector of $\Psi_{t}(g_{1}n + \sigma(1), g_{2}n + \sigma(2), \dots, g_{n}n + \sigma(n))$ is $c_{j} = c_{\sigma^{-1}(i)}$.
\end{proof}

By the previous two lemmas, the sum of the entries of the content vector of $\Psi_{t}(g_{1}n + \sigma(1), g_{2}n + \sigma(2), \dots, g_{n}n + \sigma(n))$ is
\begin{equation*}
\sum_{i=1}^{n} c_{i} = \sum_{i=1}^{n} (d_{i} + g_{i} \kappa),
\end{equation*}
which is independent of $\sigma$ and $t \in Std(\lambda)$ (the sum of the $d_{i}$ only depends on the shape of $\lambda$).  

We now define some quantities relating signs of norms of certain elements.

\begin{definition}
Let $t \in Std(\lambda)$, and $\mu = (g_{1} \leq g_{2} \leq \cdots \leq g_{n})$ be fixed.  Let 
 
\begin{equation*}
N(t;\mu) = \frac{\{ \langle t,t \rangle \}}{\langle \Psi_{t}(g_{1}n + 1, g_{2}n + 2, \dots, g_{n}n + n), \Psi_{t}(g_{1}n + 1, g_{2}n + 2, \dots, g_{n}n + n) \rangle}.
\end{equation*}
\end{definition}

\begin{definition}
Let $t \in Std(\lambda)$ and $\mu = (g_{1} \leq g_{2} \leq \cdots \leq g_{n})$ be fixed.  Let $\sigma \in S_{n}$ such that if $i<j$ and $g_{i} = g_{j}$ then $\sigma(i)<\sigma(j)$.  Let 
\begin{equation*}
N(t;\mu;\sigma) = \frac{\{ \langle \Psi_{t}(g_{1}n + 1, \dots, g_{n}n + n), \Psi_{t}(g_{1}n + 1, \dots, g_{n}n + n) \rangle \}}{\{ \langle \Psi_{t}(g_{1}n + \sigma(1), \dots, g_{n}n + \sigma(n)), \Psi_{t}(g_{1}n + \sigma(1), \dots, g_{n}n + \sigma(n)) \rangle \}}.
\end{equation*}

\end{definition}

With this notation, the signature character is
\begin{equation*}
\sum_{\substack{\mu = (g_{1} \leq g_{2} \leq \cdots \leq g_{n}) \in \mathbb{Z}_{+}^{n} \\ t \in Std(\lambda)}} \Bigg( t^{f(\lambda) + \kappa |\mu|} \{N(t;\mu) \} \sum_{\substack{\sigma \in S_{n} \\ i<j \text{ and }g_{i}=g_{j} \\ \Rightarrow \sigma(i)<\sigma(j)}} \{ N(t;\mu;\sigma) \} \Bigg) ,
\end{equation*}
where $f(\lambda) = \sum_{i=1}^{n} d_{i}$ using the notation of the previous lemma.  We will determine this more explicitly.

\begin{proposition}
Let $t \in Std(\lambda)$, $\mu = (g_{1} \leq g_{2} \leq \cdots \leq g_{n}) \in \mathbb{Z}_{+}^{n}$.  Also let $\sigma \in S_{n}$ with $i<j$ and $g_{i} = g_{j} \Rightarrow \sigma(i)<\sigma(j)$.  Let $(c_{1}, \dots, c_{n})$ be the content vector of $\Psi_{t}(g_{1}n +  1, g_{2}n + 2, \dots, g_{n}n+n)$.  Then
\begin{equation*}
 N(t;\mu;\sigma) = \prod_{\substack{s<t \\ \sigma(s) > \sigma(t)}} \Big[ (c_{s}-c_{t})^{2}-1 \Big].
\end{equation*}
\end{proposition} 
\begin{proof}
We will first determine the algorithm that starts with tableau $(g_{1}n + 1, g_{2}n + 2, \dots, g_{n}n + n)$ and produces $(g_{1}n + \sigma(1), g_{2}n + \sigma(2), \dots, g_{n}n + \sigma(n))$, via steps $T_{j,j+1}$ (swaps $j,j+1$ in the tableau and extend to classes modulo $n$).  Let $\sigma(j) = n$.  Then 
\begin{multline*}
T_{n-1,n} \cdots T_{j+1,j+2} T_{j,j+1}(g_{1}n + 1, g_{2}n + 2, \dots, g_{n}n + n) \\ = (g_{1}n+1, \dots, g_{j-1}n+(j-1), g_{j}n+n, g_{j+1}n + j, \dots, g_{n}n + (n-1)),
\end{multline*}
so that $n$ is now in the correct position.  One can also check that this is a legimitate series of moves, i.e., preserves standard tableaux conditions (increasing left to right along rows, top to bottom along columns).  The content vectors change as follows through the above steps:
\begin{multline*}
(c_{1}, \dots, c_{n}) \rightarrow (c_{1}, \dots, c_{j-1}, c_{j+1},c_{j}, \dots, c_{n}) \\
\rightarrow (c_{1}, \dots, c_{j-1}, c_{j+1}, c_{j+2}, c_{j}, c_{j+3}, \dots, c_{n}) \rightarrow \dots \rightarrow (c_{1}, \dots, c_{j-1}, c_{j+1}, \dots, c_{n}, c_{j}).
\end{multline*}
It was proved in \cite[Appendix]{ESG} that if $f$ has content vector $(\alpha_{1}, \dots, \alpha_{n})$ then (up to sign) the norms of $f$ and $\sigma_{i}f$ are related by
\begin{equation*}
\langle \sigma_{i}f, \sigma_{i}f \rangle = \langle f, f \rangle \Big[ (\alpha_{i}-\alpha_{i+1})^{2}-1 \Big].
\end{equation*}
Thus, the associated norm factor induced by these steps is
\begin{equation*}
\prod_{j<i \leq n} \Big[ (c_{j}-c_{i})^{2}-1 \Big] = \prod_{\sigma^{-1}(n)<i \leq n} \Big[ (c_{j}-c_{i})^{2}-1 \Big], 
\end{equation*}
since $j = \sigma^{-1}(n)$.  Iterating this argument, the associated norm factor from the above series of steps that starts with tableau $(g_{1}n + 1, g_{2}n + 2, \dots, g_{n}n + n)$, with content vector $(c_{1}, \dots, c_{n})$ and produces $(g_{1}n + \sigma(1), g_{2}n + \sigma(2), \dots, g_{n}n + \sigma(n))$ is
\begin{equation*}
 \prod_{\substack{s<t \\ \sigma(s) > \sigma(t)}} \Big[ (c_{s}-c_{t})^{2}-1 \Big] = N(t;g;\sigma), 
\end{equation*}
i.e., the product is over \textit{inversions} of $\sigma$.
\end{proof}

\begin{proposition}
Let $\mu = (g_{1} \leq g_{2} \leq \cdots \leq g_{n}) \in \mathbb{Z}_{+}^{n}$ and $t \in Std(\lambda)$ be fixed.  Then
\begin{equation*}
\sum_{\substack{\sigma \in S_{n}: \\ i<j \text{ and }g_{i}=g_{j} \\ \Rightarrow \sigma(i)<\sigma(j)}} N(t;\mu;\sigma) = \frac{1}{v_{\mu}(t)} \sum_{\substack{\sigma \in S_{n}}} N(t;\mu;\sigma), 
\end{equation*}
where 
\begin{equation*}
v_{\mu}(t) = \sum_{\sigma \in S(m_{1}(\mu))} N(t; 1^{m_{1}(\mu)}; \sigma) \times \sum_{\sigma \in S(m_{2}(\mu))} N(t; 1^{m_{2}(\mu)}; \sigma) \times \cdots .
\end{equation*}
\end{proposition}

\begin{proof}
Follows from the relation between $S_{n}$ and the restricted permutations $\{ \sigma \in S_{n}: i<j \text{ and } g_{i} = g_{j} \Rightarrow \sigma(i) < \sigma(j) \}$ and coefficients $N(t; \mu; \sigma)$.
\end{proof}

Thus, the signature character may be rewritten as
\begin{equation*}
\sum_{\substack{\mu = (g_{1} \leq g_{2} \leq \cdots \leq g_{n}) \in \mathbb{Z}_{+}^{n} \\ T \in Std(\lambda)}} \Bigg( t^{f(\lambda) + \kappa |\mu|} \frac{\{N(T;\mu)\}}{v_{\mu}(T)} \sum_{\substack{\sigma \in S_{n} }} \Big( \prod_{\substack{s<t \\ \sigma(s) > \sigma(t)}} \Big\{ (c_{s}-c_{t})^{2}-1 \Big\} \Big) \Bigg),
\end{equation*}
where $(d_{1}, \dots, d_{n})$ is the content vector of $T$, $f(\lambda) = \sum_{i=1}^{n} d_{i}$, and $(c_{1}, \dots, c_{n})$ is the content vector of $\Psi_{T}(g_{1}n + 1, g_{2}n + 2, \dots, g_{n}n+n)$.

\begin{lemma}
Let $\mu = (g_{1} \leq g_{2} \leq \cdots \leq g_{n}) \in \mathbb{Z}_{+}^{n}$ and $t \in Std(\lambda)$ be fixed.  Then (up to sign)
\begin{equation*}
v_{\mu}(t) = m_{1}(\mu)! m_{2}(\mu)! \cdots.
\end{equation*}
\end{lemma}
\begin{proof}
We show that (up to sign) for $i \geq 1$
\begin{equation*}
N(t; 1^{m_{i}(\mu)}; \sigma) = 1;
\end{equation*}
combined with the definition of $v_{\mu}(t)$ this would prove the claim.  Let us put $k = m_{i}(\mu)$.  Now by the previous proposition
\begin{equation*}
N(t; 1^{m_{i}(\mu)}; \sigma) = \prod_{\substack{s<t \\ \sigma(s) > \sigma(t)}} \Big[ (c_{s}-c_{t})^{2}-1 \Big],
\end{equation*}
where $(c_{1}, \dots, c_{k})$ is the content vector of $(g_{1}n + 1, \dots, g_{n}n+n)$.  But we have $\mu = (1, \dots, 1)$ in this case, so $c_{s} - c_{t} = d_{s} - d_{t}$, where $(d_{1}, \dots, d_{n})$ is the content vector of $t$.
\end{proof}

\begin{remarks}
A translation is required: if $(c_{1}, \dots, c_{n})$ is the content vector of $T$, then $(c_{n} + \kappa, c_{n-1}+\kappa, \dots, c_{1} + \kappa)$ is the weight vector with respect to $z$, and Griffeth in \cite[Appendix]{ESG} computes the norm multiples with respect to the $z$-weights. 
\end{remarks}

\begin{proposition}
Let $T \in Std(\lambda)$ with content vector $(d_{1}, \dots, d_{n})$, and $\mu = (g_{1} \leq g_{2} \leq \cdots \leq g_{n})$.  Then we have the following formula for $N(T; \mu)$
\begin{equation*}
N(T; \mu) = \prod_{i=1}^{n} \Bigg[ \prod_{j=1}^{g_{i}} (d_{i} + j\kappa) \Bigg] \Bigg[ \prod_{l=1}^{i-1} \prod_{j=1}^{g_{i}-g_{l}} \Big( (d_{i} + j\kappa - d_{l})^{2}-1\Big)\Bigg].
\end{equation*}
\end{proposition}
\begin{proof}
Use the word that takes $T$ to $(g_{1}n + 1, g_{2}n + 2, \dots, g_{n}n+n)$, and compute the associated norm factors.  
\end{proof}

Using the theorems above, we have the following formula for the signature character (after pulling out $t^{f(\lambda)}$ and substituting $t$ for $t^{\kappa}$)
\begin{multline*}
\sum_{\substack{\mu = (g_{1} \leq g_{2} \leq \cdots \leq g_{n}) \in \mathbb{Z}_{+}^{n} \\ T \in Std(\lambda) \text{ with} \\ \text{content vector }(d_{1}, \dots, d_{n})}} t^{|\mu|} \prod_{i=1}^{n} \Bigg[ \prod_{j=1}^{g_{i}} \{d_{i} + j\kappa \} \Bigg] \Bigg[ \prod_{l=1}^{i-1} \prod_{j=1}^{g_{i}-g_{l}} \Big\{ (d_{i} + j\kappa - d_{l})^{2}-1\Big\}\Bigg] \times \\
\times  \frac{1}{m_{1}(\mu)! m_{2}(\mu)! \cdots}\sum_{\substack{\sigma \in S_{n} }} \prod_{\substack{s<t \\ \sigma(s) > \sigma(t)}} \Big\{ (c_{s}-c_{t})^{2}-1 \Big\},
\end{multline*}
where $(c_{1}, \dots, c_{n}) = (d_{1} + \kappa g_{1}, \dots, d_{n} + \kappa g_{n})$.  We rewrite this as 
\begin{multline} \label{infseries}
\sum_{\substack{\mu = (g_{1} \leq g_{2} \leq \cdots \leq g_{n}) \in \mathbb{Z}_{+}^{n} \\ T \in Std(\lambda) \text{ with} \\ \text{content vector }(d_{1}, \dots, d_{n}) }} t^{|\mu|} \prod_{i=1}^{n} \Bigg[ \prod_{j=1}^{g_{i}} \{d_{i} + j\kappa\} \Bigg] \Bigg[ \prod_{l=1}^{i-1} \prod_{j=1}^{g_{i}-g_{l}} \Big\{ ((d_{i} - d_{l}) + j\kappa)^{2}-1\Big \}\Bigg] \times \\
\times  \frac{1}{m_{1}(\mu)! m_{2}(\mu)! \cdots}\sum_{\substack{\sigma \in S_{n} }} \prod_{\substack{s<t \\ \sigma(s) > \sigma(t)}} \Big \{ ((d_{t}-d_{s}) + \kappa(g_{t} - g_{s}))^{2}-1 \Big \}.
\end{multline}

\begin{proposition}
Let $(d_{1}, \dots, d_{n})$ be the content vector of $T \in Std(\lambda)$.  Also let $1 \leq l < i \leq n$, and $N \in \mathbb{Z}_{+}$.  Then the sign of 
\begin{equation*}
((d_{i}-d_{l}) +N\kappa)^{2} - 1 
\end{equation*}
is positive for 
\begin{equation*}
\begin{cases}
\text{ all } N \in \mathbb{Z}_{+}, & \text{ if } d_{i} - d_{l} > 0 \\
N \geq \lceil c(d_{i}-d_{l}) - c \rceil, N <  \lfloor c(d_{i}-d_{l}) + c \rfloor & \text{ if } d_{i}-d_{l} < 0.
\end{cases}
\end{equation*}
\end{proposition}
\begin{proof}
First note that, up to sign, we have
\begin{equation*}
((d_{i}-d_{l}) +N\kappa)^{2} - 1 = (c(d_{i}-d_{l}) - N)^{2} - c^{2} = 1 - \frac{c^{2}}{(c(d_{i}-d_{l}) - N)^{2}}.
\end{equation*}
This is negative if and only if $|(c(d_{i}-d_{l}) - N)| < |c|$.  If $d_{i}-d_{l} >0$, we have $|(c(d_{i}-d_{l}) - N)|  > |c|$ for all $N \in \mathbb{Z}_{+}$ (since $c(d_{i} - d_{l}) \leq c < 0$).  

Now suppose $d_{i}-d_{l} <0$, so $c(d_{i}-d_{l}) > 0$.  Then for $N \geq \lceil c(d_{i}-d_{l}) - c \rceil$, we have $|(c(d_{i}-d_{l}) - N)|  \geq |c|$, so 
\begin{equation*}
1 - \frac{c^{2}}{(c(d_{i}-d_{l}) - N)^{2}}
\end{equation*}
is positive.  Note that $0 < -c \leq c(d_{i}-d_{l})$, and for $0 \leq N \leq \lfloor c(d_{i}-d_{l}) + c \rfloor$, we have $|(c(d_{i}-d_{l}) - N)|  > |c|$, so the expression above is positive there as well.  It is negative for $\lfloor c(d_{i}-d_{l}) + c \rfloor \leq N \leq \lfloor c(d_{i}-d_{l}) - c \rfloor$.
\end{proof}

\begin{lemma}
Let $d_{i} - d_{l} < 0$, in the context of the previous proposition.  Then
\begin{equation*}
\Big \{  \prod_{j=1}^{K} \big((d_{i}-d_{l}) + j \kappa)^{2} - 1\big)\Big \} = (-1)^{\min \{ \lfloor c(d_{i}-d_{l}+1) \rfloor, K \} } (-1)^{\min \{ \lfloor c(d_{i}-d_{l}-1) \rfloor, K \}}
\end{equation*}
\end{lemma}

\begin{proof}
We define 
\begin{equation*}
F_{1}(j ; c; l) = \begin{cases} (-1) ,& \text{if } j \leq cl \\
1 ,& \text{if } j>cl,
\end{cases}
\end{equation*}
then $\{ ((d_{i}-d_{l}) + j\kappa)^{2} - 1 \} = F_{1}(j; c; d_{i}-d_{l} + 1) F_{1}(j;c;d_{i}-d_{l}-1)$, by the previous proposition.  Taking the product of these terms over $1 \leq j \leq K$ gives the result.
\end{proof}

\begin{definition}
Let $T \in Std(\lambda)$ with content vector $(d_{1}, \dots, d_{n})$.  For $1 \leq l<i \leq n$, we define
\begin{equation*}
c_{(l,i)} = \begin{cases} 1 ,& \text{ if } d_{i} - d_{l} > 0 \\
(-1)^{\lfloor c(d_{i}-d_{l}-1) \rfloor - \lfloor c(d_{i}-d_{l}+1) \rfloor} ,& \text{ if } d_{i} - d_{l} < 0.
\end{cases}
\end{equation*}
\end{definition}
\begin{remarks}
We note that, by the previous proposition, $c_{(l,i)}$ is the sign of the product
\begin{equation*}
\prod_{N = 1}^{K} \Bigg( \big((d_{i} - d_{l}) + N \kappa \big)^{2} - 1 \Bigg) = \prod_{N = 1}^{K} \Bigg( 1 - \frac{c^{2}}{(c(d_{i}-d_{l})-N)^{2}} \Bigg),
\end{equation*}
for any $K \geq \lfloor c(d_{i} - d_{l}-1)  \rfloor$.  
\end{remarks}

\begin{proposition}
Let $T \in Std(\lambda)$ with content vector $(d_{1}, \dots, d_{n})$.  For $1 \leq l<i \leq n$, we have
\begin{equation*}
c_{(l,i)} = \begin{cases} 1 ,& \text{ if } d_{i}-d_{l} > 0 \\
 \Big \{ \frac{\sin(\pi c(d_{i}-d_{l}+1))}{\pi c(d_{i}-d_{l}+1)}\frac{\sin(\pi c(d_{i}-d_{l}-1))}{\pi c(d_{i}-d_{l}-1)} \Big \},& \text{ if } d_{i} - d_{l} < 0.
\end{cases}
\end{equation*}
\end{proposition}
\begin{proof}
We consider the case $d_{i}-d_{l} < 0$.  We use Euler's identity for sine
\begin{equation*}
\prod_{j\geq 1} \Big( 1 - \frac{z^{2}}{j^{2}} \Big) = \frac{\sin(\pi z)}{\pi z}.
\end{equation*}
Taking signs of the left and right hand side yields
\begin{equation*}
(-1)^{\lfloor |z| \rfloor} = \Big\{ \frac{\sin(\pi z)}{\pi z} \Big\}.
\end{equation*}
Thus,
\begin{multline*}
c_{(l,i)} = (-1)^{\lfloor c(d_{i}-d_{l}-1) \rfloor} (-1)^{\lfloor c(d_{i}-d_{l}+1) \rfloor} 
 = \Big \{ \frac{\sin(\pi c(d_{i}-d_{l}-1)}{\pi c(d_{i}-d_{l}-1)} \Big \} \Big \{ \frac{\sin(\pi c(d_{i}-d_{l} + 1)}{\pi c(d_{i}-d_{l} + 1)}  \Big \},
\end{multline*}
as desired.  
\end{proof}

\begin{proposition}
Let $(d_{1}, \dots, d_{n})$ be the content vector of $T \in Std(\lambda)$.  Let $1 \leq i \leq n$ and $j \in \mathbb{Z}_{+}$.  Then
\begin{equation*}
d_{i} + j\kappa = \frac{cd_{i} - j}{c}
\end{equation*}
is positive for $j \geq \lceil cd_{i} \rceil$ and negative $0 \leq j \leq \lfloor cd_{i} \rfloor$ if $d_{i} < 0$.  If $d_{i} > 0$, it is always positive.
\end{proposition}
\begin{proof}
Obvious, using $c<0$.  
\end{proof}

\begin{definition}
Let $T \in Std(\lambda)$ with content vector $(d_{1}, \dots, d_{n})$.  For $1 \leq i \leq n$, we define 
\begin{equation*}
c_{i} = (-1)^{\lfloor cd_{i} \rfloor}.
\end{equation*}
\end{definition}
\begin{remarks}
 We note that, by the previous proposition $c_i$ is the infinite product of signs:
\begin{equation*}
\prod_{j \geq 1} \{(d_i + j\kappa)\}.
\end{equation*}
Note that the terms are eventually all equal to 1, so the product is indeed convergent.
\end{remarks}

\begin{proposition} \label{inffrmla}
Let $\lambda$ be a partition with $|\lambda| = n$, and $c<0$.  Then $ch_{s}(M_{c}(\lambda))$ is equal to
\begin{multline*}
\sum_{\substack{\mu = (g_{1} \leq g_{2} \leq \cdots \leq g_{n}) \in \mathbb{Z}_{+}^{n} \\ T \in Std(\lambda) \text{ with} \\ \text{content vector }(d_{1}, \dots, d_{n}) }} \Bigg( t^{|\mu|} \prod_{i=1}^{n} (-1)^{\min\{g_{i}, \lfloor cd_{i} \rfloor\}} \\
\times   \sum_{\substack{\sigma \in R }} \Bigg( \prod_{\substack{1 \leq s<t \leq n\\ \sigma(s) > \sigma(t)}} (-1)^{\min\{ (g_{t}-g_{s}-1, \lfloor c(d_{t}-d_{s}-1) \rfloor \}} (-1)^{\min\{ (g_{t}-g_{s}-1, \lfloor c(d_{t}-d_{s}+1) \rfloor \}} \\
\times \prod_{\substack{1 \leq s<t \leq n \\ \sigma(s) < \sigma(t)}}(-1)^{\min\{ (g_{t}-g_{s}, \lfloor c(d_{t}-d_{s}-1) \rfloor \}} (-1)^{\min\{ (g_{t}-g_{s}, \lfloor c(d_{t}-d_{s}+1) \rfloor \}} \Bigg) \Bigg),
\end{multline*}
where it is understood that if the $\min$-functions above have a negative argument, $(-1)^{\min \{ \cdot, \cdot \}}$ is equal to one, and $R \subset S_{n}$ is the set of restricted permutations:
\begin{equation*}
\{ \sigma \in S_{n}: \text{ if } i<j \text { and } g_{i} = g_{j} \text{ then } \sigma(i) < \sigma(j) \}.
\end{equation*}
\end{proposition}

\begin{proof}
Follows from equation (\ref{infseries}) as well as the sign computations of the functions $[d_{i} + j\kappa]$, $[((d_{i}-d_{l}) + j\kappa)^{2}-1]$ as in the previous theorems. 
\end{proof}

\begin{definition} \label{functions}
For $v=(\mu, \sigma, T) \in \mathcal{B}(M_{c}(\lambda))$ as in the sum of the previous proposition, we let
\begin{equation*}
wt(v) = |\mu|
\end{equation*}
and
\begin{multline*}
f(v) = \sum_{i=1}^{n} \min\{g_{i}, \lfloor cd_{i} \rfloor\} \\ + \sum_{\substack{1 \leq s<t \leq n \\ \sigma(s) > \sigma(t)}} \min\{ (g_{t}-g_{s}-1, \lfloor c(d_{t}-d_{s}-1) \rfloor \} + \min\{ (g_{t}-g_{s}-1, \lfloor c(d_{t}-d_{s}+1) \rfloor \}  \\+ \sum_{\substack{1 \leq s<t \leq n \\ \sigma(s) < \sigma(t)}} \min\{ (g_{t}-g_{s}, \lfloor c(d_{t}-d_{s}-1) \rfloor \} + \min\{ (g_{t}-g_{s}, \lfloor c(d_{t}-d_{s}+1) \rfloor \}.
\end{multline*}
\end{definition}

We will now prove Theorem \ref{RCAinfsum}, mentioned in the Introduction of the paper.

\begin{proof}[Proof of Theorem \ref{RCAinfsum}]
Follows from Proposition \ref{inffrmla} and Definition \ref{functions}.
\end{proof}

\begin{definition}
Let $T \in Std(\lambda)$ with content vector $(d_{1}, \dots, d_{n})$.  We define 
\begin{equation*}
N_{max}(T) = \max \Big( \{ \lceil c(d_{i}-d_{l}) - c \rceil \}_{\substack{1 \leq l < i \leq n \\ d_{i} - d_{l} < 0}} \cup \{ \lceil cd_{i} \rceil \}_{1 \leq i \leq n}\Big).
\end{equation*}
\end{definition}
Note that we can define $N_{\max}(\lambda)$ which only depends on $\lambda$, such that $N_{\max}(\lambda) \geq N_{\max}(T)$.

\begin{definition}
For $S \subset \{ (0,1), (1,2), (2,3), \dots, (n-1,n) \}$, we let $S^{c} = \{ (0,1), (1,2), (2,3), \dots$, $(n-1, n) \} \setminus S$.  We define 
\begin{equation*}
\hat{S} = \{ (l,i): 1 \leq l < i \leq n \text{ and } (l, l+1), (l+1, l+2), \dots, (i-1, i) \not \in S \}
\end{equation*}
and
\begin{equation*}
\hat{S}^{c} = \{ (i,j): 1 \leq i<j \leq n \} \setminus \hat{S}.
\end{equation*}
Also let $i_{S} = \max \{ 0 \leq i \leq n: (0,1), (1,2), \dots, (i-1,i) \not \in S \}$ with the convention that $i_{S} = 0$ means $(0,1) \in S$.  
\end{definition}

\begin{theorem}
Let $\lambda$ be a partition such that $|\lambda| = n$, and let $c<0$.  Then $ch_{s}(M_{c}(\lambda))$ is equal to
\begin{multline*}
 \sum_{\substack{S \subset \{ (0,1), (1,2), \dots, (n-1,n) \} \\ 0 \leq D_{j} \leq N_{\max}(\lambda): (j,j+1) \not \in S }} \frac{t^{\sum_{(j,j+1) \not \in S} (n-j)D_{j}} t^{N_{\max}(\lambda) \sum_{(j,j+1) \in S} (n-j)}}{(1-t)^{|S|}}   \\
\times \sum_{\substack{T \in Std(\lambda) \\ (d_{1}, \dots, d_{n}) \text{ content vector of }T}}  \prod_{j=1}^{i_{S}} (-1)^{\min\{ D_{0} + D_{1} + \cdots + D_{j-1} , \lfloor cd_{j} \rfloor \} }  \prod_{j=i_{S}+1}^{n} c_{j} \prod_{(l,i) \in \hat{S}^{c}} c_{(l,i)} \\
\times  \sum_{\sigma \in R} \Bigg( \prod_{\substack{(s,t) \in \hat{S} \\ \sigma(s) > \sigma(t) \\ \text{(inversion terms)}}} (-1)^{\min\{ D_{s} + \cdots + D_{t-1}-1, \lfloor c(d_{t}-d_{s}-1) \rfloor \}} (-1)^{\min \{ D_{s} + \cdots + D_{t-1}-1, \lfloor c(d_{t}-d_{s}+1) \rfloor \}} \\
\times \prod_{\substack{(s,t) \in \hat{S} \\ \sigma(s) < \sigma(t)}} (-1)^{\min \{ D_{s} + \cdots + D_{t-1}, \lfloor c(d_{t}-d_{s}-1) \rfloor \}} (-1)^{\min \{ D_{s} + \cdots + D_{t-1}, \lfloor c(d_{t}-d_{s}+1) \rfloor \}} \Bigg).
\end{multline*}
\end{theorem}
\begin{proof}
Our starting point is the formula for the signature character provided in Proposition \ref{inffrmla}.  We will reparametrize by using
\begin{equation*}
 D_{0} = g_{1} - 0, D_{1} = g_{2} - g_{1}, D_{2} = g_{3} - g_{2}, \dots, D_{n-1} = g_{n} - g_{n-1}.  
\end{equation*}
So we have, for $\mu = (g_{1} \leq g_{2} \leq \cdots \leq g_{n})$, 
\begin{multline*}
|\mu| = g_{1} + \cdots + g_{n} = g_{1} + (g_{1} + (g_{2} - g_{1})) + (g_{1} + (g_{2} - g_{1}) + (g_{3} - g_{2})) + \cdots \\
= ng_{1} + (n-1)(g_{2} - g_{1}) + (n-2)(g_{3} - g_{2}) + \cdots + 2(g_{n-1}-g_{n-2}) + (g_{n} - g_{n-1}) \\
= ng_{1} + \sum_{1 \leq i \leq n-1} (n - i) (g_{i+1}-g_{i}) = \sum_{0 \leq j \leq n-1} (n-j) D_{j}
\end{multline*}
and for $1 \leq i \leq n$, we have
\begin{equation*}
g_{i} = g_{1} + (g_{2} - g_{1}) + (g_{3} - g_{2}) + \dots + (g_{i} - g_{i-1}) = D_{0} + D_{1} + \cdots + D_{i-1}
\end{equation*}
and for $1 \leq l<i \leq n$, we have
\begin{equation*}
g_{i} - g_{l} = (g_{l+1} - g_{l}) + (g_{l+2} - g_{l+1}) + \dots + (g_{i}-g_{i-1}) = D_{l} + D_{l+1} + \cdots + D_{i-1}.
\end{equation*}
We split variables up by $0 \leq D_{j} \leq N_{\max}(\lambda)$ and $D_{j} > N_{\max}(\lambda)$.  We also use that, if $X$ is larger than a certain threshold, 
\begin{equation*}
(-1)^{\min\{X, \lfloor cd_{j} \rfloor} \} = (-1)^{\lfloor cd_{j} \rfloor}
\end{equation*}
\begin{equation*}
(-1)^{\min \{X, \lfloor c(d_{t}-d_{s}-1) \rfloor \}} = (-1)^{\lfloor c(d_{t}-d_{s}-1) \rfloor},
\end{equation*}
etc.

\end{proof}

We note that the previous result implies that $ch_{s}(M_{c}(\lambda))$ is a sum of $2^{n}$ terms, where each term is a rational function in $t$; it also implies Corollary \ref{rationalfcn} of the Introduction.

\begin{theorem}
Let $m \in \mathbb{Z}_{+}$, $\lambda$ a partition with $|\lambda| = n$, and $c<0$.  Then $ch_{s}(L_{c-m}(\lambda))$ is a rational function of $t, t^{m}, (-1)^{m}$.  
\end{theorem}
\begin{proof}
We start with the formula of Proposition \ref{inffrmla}, in terms of the $g_{i}$ variables.  Consider the $n$-dimensional space defined by $g_{1} \leq g_{2} \leq \cdots \leq g_{n} \in \mathbb{Z}_{+}^{n}$.  Consider all hyperplanes 
\begin{equation*}
g_{i} = cd_{i}
\end{equation*}
\begin{equation*}
g_{t} = g_{s} + 1+ c(d_{t}-d_{s} \pm 1)
\end{equation*}
\begin{equation*}
g_{t} = g_{s} + c(d_{t}-d_{s} \pm 1),
\end{equation*}
where $1 \leq i \leq n$ and $1 \leq s<t \leq n$ and $d_{i}<0, d_{t}-d_{s} < 0$.  One can split it up into regions such that each region is defined by $l_{1}(g_{1}, \dots, g_{i-1}; d; c) \leq g_{i} \leq l_{2}(g_{1}, \dots, g_{i-1};d;c)$, where $l_{1}, l_{2}$ are linear functions in the variables.  One can argue that on each region, we have from the contribution of the $\min$ functions (for variable $g_{i}$), 
\begin{equation*}
(-1)^{l_{3}(g_{1}, \dots, g_{i-1}; d; c)} (-1)^{l_{4}(g_{i})},
\end{equation*}
where $l_{3}$, $l_{4}$ are linear functions.  Here $d$ denotes the content vector $(d_{1}, \dots, d_{n})$.  Thus at each step (starting from variable $g_{n}$ and iteratively working down to $g_{1}$), the summation would be of the form
\begin{equation*}
\sum_{g_{i} = l_{1}(g_{1}, \dots, g_{i-1};d;c)}^{l_{2}(g_{1} , \dots, g_{i-1};d;c)} t^{g_{i}} (-1)^{l_{3}(g_{1}, \dots, g_{i-1};d;c)} (-1)^{l_{4}(g_{i})}.
\end{equation*}
Then one can use the identity
\begin{equation*}
\sum_{i=a}^{b} t^{i} (-1)^{ki} = \begin{cases} \frac{t^{b}-t^{a}}{1-t} ,& \text{if } $k$ \text{ is even} \\
\frac{(-t)^{b} - (-t)^{a}}{1+t} ,& \text{if } $k$ \text{ is odd. }
\end{cases}
\end{equation*}

\end{proof}

We now turn to investigating the connection between signatures of $\mathcal{H}_{n}(q)$ and signature characters of $\mathbb{H}_{c}$.  Recall the definition of the asymptotic character, $a_{s}(M_{c}(\lambda))$, found in Equation (\ref{asymchar}) of the Introduction.

Note that, in the formula for $ch_{s}(M_{c}(\lambda))$ as a sum of $2^{n}$ terms, it is obtained by evaluating the term corresponding to $S = \{1,2, \cdots, n \}$ at $t=1$. 

\begin{remarks}
We recall that the module $M_{c}(\lambda)$ is $\mathbb{C}[x_{1}, \dots, x_{n}] \otimes S^{\lambda}$ as a vector space.  The degree $m$ homogeneous subspace is spanned by $m_{\mu}(x) \otimes f_{T}$, where $m_{\mu}$ are monomials of degree $m$ and $f_{T}$ is a basis for $S^{\lambda}$.  The coefficient on $t^{m}$ in $ch_{s}$ is equal to the dimension of elements in the graded $m$ subspace with positive norm, minus those with negative norm in the same subspace; let this coefficient be denoted by $s_{m}$.  Now, the dimension of the whole subspace is 
\begin{equation*}
\dim S^{\lambda}\times \# \{\text{monomials in } n \text{ variables of degree }m \},
\end{equation*}
the first quantity is given by the hook length formula, and the second is the binomial coefficient $c_{m}$ giving the coefficient on $t^{m}$ in $1/(1-t)^{n}$.  One can show (via a power series argument) that 
\begin{equation*}
\lim_{m \rightarrow \infty} \frac{s_{m}}{c_{m}} = p(1; \lambda, c).
\end{equation*}
\end{remarks}

\begin{theorem}
Let $\lambda$ be fixed, with $|\lambda| = n$ and put $q = e^{2\pi i c}$.  The asymptotic signature character is
\begin{equation*}
\sum_{\substack{T \in Std(\lambda) \\ (d_{1}, \dots, d_{n}) \text{ content vector of } T}} \prod_{1 \leq i \leq n} \prod_{\substack{1 \leq l \leq i-1 \\ d_{i} - d_{l} < 0}} \{[[d_{i}-d_{l} + 1]]_{\sqrt{q}}\} \{[[d_{i}-d_{l} - 1 ]]_{\sqrt{q}}\}.
\end{equation*}
\end{theorem}
\begin{proof}
To obtain the asymptotic part, we multiply the signature character by $(1-t)^{n}$ and evaluate at $t=1$.  Using the previous theorem, this is the same as taking the term in the sum of $2^{n}$ terms corresponding to $S = \{ (0,1), (1,2), \dots, (n-1,n) \}$, multiplying by $(1-t)^{n}$ and evaluating at $t=1$.  For this choice of $S$, we have $i_{S} = 0$ and $\hat{S} = \emptyset$.  So we obtain for the asymptotic signature character
\begin{multline*}
\sum_{\substack{T \in Std(\lambda) \\ (d_{1}, \dots, d_{n}) \text{ content vector of }T}} \prod_{j=1}^{n} c_{j} \prod_{1 \leq l<i \leq n} c_{(l,i)} \\
= \sum_{\substack{T \in Std(\lambda) \\ (d_{1}, \dots, d_{n}) \text{ content vector of }T}} \prod_{\substack{j=1 \\ d_{j} < 0}}^{n} (-1)^{\lfloor cd_{j} \rfloor} \prod_{\substack{1 \leq l< i \leq n \\ d_{i} - d_{l} < 0}} \Big \{ \frac{\sin(\pi c(d_{i}-d_{l}-1)}{\pi c(d_{i}-d_{l}-1)} \Big \} \Big \{ \frac{\sin(\pi c(d_{i}-d_{l} + 1)}{\pi c(d_{i}-d_{l} + 1)}  \Big \} \\
= \prod_{\substack{j=1 \\ d_{j} < 0}}^{n} (-1)^{\lfloor cd_{j} \rfloor} \sum_{\substack{T \in Std(\lambda) \\ (d_{1}, \dots, d_{n}) \text{ content vector of }T}}  \prod_{\substack{1 \leq l< i \leq n\\ d_{i}-d_{l}<0}} \Big \{ \frac{\sin(\pi c(d_{i}-d_{l}-1)}{\pi c(d_{i}-d_{l}-1)} \Big \} \Big \{ \frac{\sin(\pi c(d_{i}-d_{l} + 1)}{\pi c(d_{i}-d_{l} + 1)}  \Big \}.
\end{multline*}
Now we let $q = e^{2\pi i c}$, so that
\begin{equation*}
\Big \{ \frac{\sin(\pi c(d_{i}-d_{l}-1)}{\pi c(d_{i}-d_{l}-1)} \Big \} = \Big \{ \frac{e^{\pi i c(d_{i}-d_{l}-1)} - e^{-\pi i c(d_{i}-d_{l}-1)}}{2 i \pi c(d_{i}-d_{l}-1)} \Big \} = \Big \{ \frac{\sqrt{q}^{(d_{i}-d_{l}-1)} - \sqrt{q}^{-(d_{i}-d_{l}-1)}}{2 i \pi c(d_{i}-d_{l}-1)} \Big \}
\end{equation*}
and similarly,
\begin{equation*}
\Big \{ \frac{\sin(\pi c(d_{i}-d_{l}+1)}{\pi c(d_{i}-d_{l}+1)} \Big \} = \Big \{ \frac{\sqrt{q}^{(d_{i}-d_{l}+1)} - \sqrt{q}^{-(d_{i}-d_{l}+1)}}{2 i \pi c(d_{i}-d_{l}+1)} \Big \}.
\end{equation*}
Thus, we have
\begin{multline*}
 \prod_{\substack{1 \leq l< i \leq n \\ d_{i} - d_{l} < 0}} \Big \{ \frac{\sin(\pi c(d_{i}-d_{l}-1)}{\pi c(d_{i}-d_{l}-1)} \Big \} \Big \{ \frac{\sin(\pi c(d_{i}-d_{l} + 1)}{\pi c(d_{i}-d_{l} + 1)}  \Big \} \\
= \prod_{\substack{1 \leq l< i \leq n \\ d_{i} - d_{l} < 0}}  \Big \{ \frac{(\sqrt{q}^{(d_{i}-d_{l}-1)} - \sqrt{q}^{-(d_{i}-d_{l}-1)})}{\sqrt{q} - \sqrt{q}^{-1}} \frac{(\sqrt{q}^{(d_{i}-d_{l}+1)} - \sqrt{q}^{-(d_{i}-d_{l}+1)})}{\sqrt{q} - \sqrt{q}^{-1}} \Big \} \\
= \prod_{\substack{1 \leq l< i \leq n \\ d_{i} - d_{l} < 0}} \{[[ d_{i}-d_{l}-1 ]]_{\sqrt{q}}\} \{[[d_{i}-d_{l}+1 ]]_{\sqrt{q}}\} = \prod_{\substack{1 \leq l< i \leq n \\ d_{i} - d_{l} > 0}} \{[[ d_{i}-d_{l}-1 ]]_{\sqrt{q}}\} \{[[d_{i}-d_{l}+1 ]]_{\sqrt{q}}\}, 
\end{multline*}
where for the last equality we have used that for any $N$
\begin{multline*}
[[N-1]]_{q} [[N+1]]_{q} = \frac{q^{N-1} - q^{-(N-1)}}{q - q^{-1}} \frac{q^{N+1}- q^{-(N+1)}}{q-q^{-1}} \\
= \frac{q^{-(N-1)} - q^{(N-1)}}{q-q^{-1}} \frac{q^{-(N+1)} - q^{N+1}}{q-q^{-1}} = [[-N+1]]_{q} [[-N-1]]_{q}.
\end{multline*}

\end{proof}

Finally, we use the previous result which computes $a_{s}(M_{c}(\lambda))$ explicitly to prove Theorem \ref{RCAHecke} of the Introduction.

\begin{proof}[Proof of Theorem \ref{RCAHecke}]
One can compare the formula for the asymptotic signature character with that of Theorem \ref{Heckesig} for $S_{\lambda}(q)$.

\end{proof}

\section{The asymptotic limit $c \rightarrow -\infty$}

In this section, we study the asymptotic limit
\begin{equation*}
\lim_{c \rightarrow -\infty} ch_{s}(M_{c}(\tau)) = \lim_{c \rightarrow \infty} ch_{s}(M_{c}(\tau')).
\end{equation*}
We obtain a simpler expression in this limiting case, in terms of certain statistics on permutations and content vectors.

\begin{proposition}
Let $|\lambda| = n$.  We have the asymptotic limit
\begin{multline} \label{asym}
\lim_{c \rightarrow -\infty} ch_{s}(M_{c}(\lambda)) \\= \sum_{\substack{T \in Std(\lambda) \\ \text{with content } (d_{1}, \dots, d_{n})}} \frac{1}{1- (-1)^{k_{0}(T)}t^{n}}\sum_{S \subseteq \{1,2, \dots, n-1 \}} c(T,S) \prod_{j \in S} \frac{(-1)^{k_{j}(T)} t^{n-j}}{1 - (-1)^{k_{j}(T)}t^{n-j}},
\end{multline}
where $k_{j}(T)$ is a statistic only depending on the content vector of the tableaux $T$:
\begin{equation*}
k_{j}(T) = \# \{ i: d_{i} < 0 \text{ and } i>j \} + \# \{ (s,t) : 0 \neq s \leq j<t \text{ and } d_{t}-d_{s} = 0,-1 \}
\end{equation*}
and $c(T,S) \in \mathbb{Z}$ is defined as follows:
\begin{equation*}
c(T,S) = \sum_{\sigma \in R(S)} (-1)^{\# \{ \text{inversions $(s,t)$ of $\sigma$ which satisfy } d_{t}-d_{s} = 0,-1 \}} .
\end{equation*}
The set of permutations $R(S) \subset S_{n}$ in the sum above is defined as follows: if $S = \{i_{1}< i_{2} \dots <i_{k} \}$ then $\sigma \in R(S)$ satisfies $\sigma(l) < \sigma(l')$ for $i_{j} +1 \leq l < l' \leq i_{j+1}$.

\end{proposition}

\begin{proof}
We recall the result of Proposition \ref{inffrmla}:
\begin{multline*}
ch_{s}(M_{c}(\lambda)) = \sum_{\substack{\mu = (g_{1} \leq g_{2} \leq \cdots \leq g_{n}) \in \mathbb{Z}_{\geq 0}^{n} \\ T \in Std(\lambda) \text{ with} \\ \text{content vector }(d_{1}, \dots, d_{n}) }} \Bigg( t^{|\mu|} \prod_{i=1}^{n} (-1)^{\min\{g_{i}, \lfloor cd_{i} \rfloor\}} \\
\times   \sum_{\substack{\sigma \in R }} \Bigg( \prod_{\substack{1 \leq s<t \leq n\\ \sigma(s) > \sigma(t)}} (-1)^{\min\{ g_{t}-g_{s}-1, \lfloor c(d_{t}-d_{s}-1) \rfloor \}} (-1)^{\min\{ g_{t}-g_{s}-1, \lfloor c(d_{t}-d_{s}+1) \rfloor \}} \\
\times \prod_{\substack{1 \leq s<t \leq n \\ \sigma(s) < \sigma(t)}}(-1)^{\min\{ g_{t}-g_{s}, \lfloor c(d_{t}-d_{s}-1) \rfloor \}} (-1)^{\min\{ g_{t}-g_{s}, \lfloor c(d_{t}-d_{s}+1) \rfloor \}} \Bigg) \Bigg)
\end{multline*}
(with the convention that if the argument of a min-function is negative, the min is equal to $1$).
We reparametrize by summing over $D_{i} = g_{i+1} - g_{i}$ for $0 \leq i \leq n-1$ with $g_{0} = 0$ so $D_{0} = g_{1}$.  We compute
\begin{equation*}
g_{i} = D_{0} + D_{1} + \cdots + D_{i-1}
\end{equation*}
and for $s<t$
\begin{equation*}
g_{t}-g_{s} = D_{t-1} + D_{t-2} + \cdots + D_{s}
\end{equation*}
and
\begin{equation*}
g_{1} + \cdots + g_{n} = nD_{0} + (n-1)D_{1} + \cdots + D_{n-1}.
\end{equation*}
We note that 
\begin{multline*}
\lim_{c \rightarrow -\infty} \prod_{i=1}^{n} (-1)^{\min\{g_{i}, \lfloor cd_{i} \rfloor\}} = \lim_{c \rightarrow -\infty} \prod_{i=1}^{n} (-1)^{\min\{D_{0} + D_{1} + \cdots + D_{i-1}, \lfloor cd_{i} \rfloor\}} \\= \prod_{\substack{1 \leq i \leq n \\ \text{with }d_{i} < 0}} (-1)^{D_{0} + D_{1} + \cdots + D_{i-1}}.
\end{multline*}
Also
\begin{multline*}
\lim_{c \rightarrow -\infty} (-1)^{\min\{ g_{t}-g_{s}-1, \lfloor c(d_{t}-d_{s}-1) \rfloor \}} (-1)^{\min\{ g_{t}-g_{s}-1, \lfloor c(d_{t}-d_{s}+1) \rfloor \}} \\ = \begin{cases}
(-1)^{g_{t}-g_{s} - 1}, & \text{if } d_{t}-d_{s} = 0,-1 \\
1, & \text{else}
\end{cases}
\end{multline*}
and similarly
\begin{multline*}
\lim_{c \rightarrow -\infty} (-1)^{\min\{ g_{t}-g_{s}, \lfloor c(d_{t}-d_{s}-1) \rfloor \}} (-1)^{\min\{ g_{t}-g_{s}, \lfloor c(d_{t}-d_{s}+1) \rfloor \}} = \begin{cases}
(-1)^{g_{t}-g_{s} }, & \text{if } d_{t}-d_{s} = 0,-1 \\
1, & \text{else.}
\end{cases}
\end{multline*}
Thus, one can rewrite the sum over $R$ above as
\begin{multline*}
\prod_{\substack{s<t \\ d_{t}-d_{s} = 0,-1}} (-1)^{g_{t}-g_{s}} \Bigg(\sum_{\sigma \in R} \prod_{\substack{s<t \\ \sigma(s)>\sigma(t) \\ d_{t}-d_{s} = 0,-1}} (-1) \Bigg) \\= \Bigg(\sum_{\sigma \in R} (-1)^{\# \{ \text{inv $(s,t)$ of $\sigma$ which satisfy } d_{t}-d_{s} = 0,-1 \}} \Bigg)\prod_{\substack{s<t \\ d_{t}-d_{s} = 0,-1}} (-1)^{D_{t-1} + D_{t-2} + \cdots + D_{s}}. 
\end{multline*}
Note also that $R$ only depends on multiplicities of $\mu = (g_{1}, \dots, g_{n})$, or equivalently, which $D_{i}$ are equal to zero.  We will sum over subsets $S$ that pick out which $D_{i}$ ($1 \leq i \leq n-1$) are nonzero.  That is, if $S = \{i_{1}, \dots, i_{k} \}$ then $0 \neq D_{i_{1}} = g_{i_{1} + 1} - g_{i_{1}}, \dots, 0 \neq D_{i_{k}} = g_{i_{k}+1} - g_{i_{k}}$.  One can check that the condition for the set $R$ becomes the statement above for $R(S)$.

Putting this all together gives
\begin{multline*}
\lim_{c \rightarrow -\infty} ch_{s}(M_{c}(\lambda)) = \sum_{\substack{T \in Std(\lambda) \\ \text{content } (d_{1}, \dots, d_{n})}} \sum_{S \subseteq \{1,2, \dots, n-1 \}} c(T,S) \\ \times \sum_{\substack{D_{s}, s \in S \\ D_{s} > 0 \\ D_{0} \geq 0}} t^{nD_{0}} (-1)^{m_{0}D_{0}} \prod_{j \in S} t^{(n-j)D_{j}} \prod_{\substack{j \in S}} (-1)^{m_{j}D_{j}} \prod_{\substack{j \in S}} (-1)^{n_{j}D_{j}}, 
\end{multline*}
where $m_{j} = \# \{d_{i} < 0 : i>j \}$ and $n_{j} = \# \{ (s,t): 0 \neq s \leq j<t, d_{t}-d_{s} = 0,-1 \}$.  Note that $m_{j} + n_{j} = k_{j}$, as defined in the statement of the theorem.
Finally, we use
\begin{equation*}
\sum_{D_{j} > 0} (-1)^{k_{j}D_{j}} t^{(n-j)D_{j}} = \frac{(-1)^{k_{j}}t^{n-j}}{1-(-1)^{k_{j}}t^{n-j}}
\end{equation*}
and
\begin{equation*}
\sum_{D_{0} \geq 0} (-1)^{k_{0}D_{0}} t^{nD_{0}} = \frac{1}{1- (-1)^{k_{0}}t^{n}}
\end{equation*}
to simplify; this gives the result.
\end{proof}

We consider the case $\lambda = (1^{n})$, the sign permutation.  There is only one tableaux, namely one column with $1,2, \dots, n$ in the cells from top to bottom.  The content vector of this tableaux is $(d_{1}, \dots, d_{n}) = (0,-1, \dots, 1-n)$.  We simplify the formula (\ref{asym}) above.  In this case, we have
\begin{equation*}
m_{j} = \# \{ i: d_{i}, i>j \} = n-1-j
\end{equation*}
and
\begin{equation*}
n_{j} = \# \{ (s,t): 0 \neq s \leq j < t, d_{t}-d_{s} = 0,-1 \} = \begin{cases} 1, & j \neq 0 \\
0, & j=0.
\end{cases}
\end{equation*}
So for $D_{0}$ we get the function
\begin{equation*}
\frac{1}{1-(-1)^{n-1}t^{n}}  = \frac{1}{1 + (-t)^{n}}
\end{equation*}
and for $D_{j}$ with $j \neq 0$ and $j \in S$, we get the function
\begin{equation*}
\frac{(-1)^{n-j}t^{n-j}}{1-(-1)^{n-j}t^{n-j}} = \frac{(-t)^{n-j}}{1-(-t)^{n-j}}.
\end{equation*}
So we have
\begin{equation}\label{signptn}
\lim_{c \rightarrow -\infty} ch_{s}(M_{c}(1^{n})) = \frac{1}{1+(-t)^{n}} \sum_{S \subseteq \{1, \dots, n-1 \}} c(S) \prod_{j \in S} \frac{(-t)^{n-j}}{1-(-t)^{n-j}},
\end{equation}
where 
\begin{equation*}
c(S) = \sum_{\sigma \in R(S)} (-1)^{\# \{s: \sigma(s) > \sigma(s+1) \}}
\end{equation*}
and $R(S)$ is as in the statement of the theorem.

\begin{proposition}
Let $n$ be fixed.  We have
\begin{multline} \label{signptnsimp}
\lim_{c \rightarrow -\infty} ch_{s}(M_{c}(1^{n})) = \lim_{c \rightarrow \infty} ch_{s}(M_{c}(n)) \\ = \Bigg(\frac{1}{1+ (-t)^{n}} \prod_{i=1}^{n-1} \frac{1}{1-(-t)^{n-i}} \Bigg) \sum_{\sigma \in S(n)} \Bigg((-1)^{|\mathcal{I}(\sigma)|} \prod_{i \in \mathcal{I}(\sigma)} (-t)^{n-i}\Bigg),
\end{multline}
where $\mathcal{I}(\sigma) = \{ 1 \leq i \leq n-1: \sigma(i)> \sigma(i+1)\}$ is the descent set of $\sigma$.
\end{proposition}
\begin{proof}
We start with (\ref{signptn}).  Interchanging the order of summations, we rewrite this as
\begin{multline*}
\lim_{c \rightarrow -\infty} ch_{s}(M_{c}(1^{n})) = \frac{1}{1+(-t)^{n}} \sum_{\sigma \in S(n)} (-1)^{\#\{ s: \sigma(s) > \sigma(s+1)\}} \sum_{\substack{S \subseteq \{1, \dots, n-1\}:\\ \sigma \in R(S)}} \prod_{j \in S} \frac{(-t)^{n-j}}{1-(-t)^{n-j}}.
\end{multline*}
Let $\sigma \in S(n)$ be fixed.  Suppose $\{ i_{1}, \dots, i_{k} \}$ is the set of all indices such that $\sigma(i_{l}) >\sigma(i_{l}+1)$.  Then we have
\begin{multline*}
\Big\{ S \subseteq \{1, \dots, n-1\}: \sigma \in R(S) \Big\} = \Big\{\{i_{1}, \dots, i_{k}\} \cup E : E \subseteq \{1, \dots, n-1\} \setminus \{i_{1}, \dots, i_{k}\}\Big \}.
\end{multline*}
We use this to rewrite the limit as
\begin{multline*}
\frac{1}{1+(-t)^{n}} \sum_{\sigma \in S(n)} (-1)^{k} \prod_{l=1}^{k} \frac{(-t)^{n-i_{l}}}{1-(-t)^{n-i_{l}}} \sum_{\substack{E \subseteq \{1, \dots, n-1\} \setminus \{i_{1}, \dots, i_{k} \}}} \prod_{j \in E} \frac{(-t)^{n-j}}{1-(-t)^{n-j}},
\end{multline*}
where $\{i_{1}, \dots, i_{k} \}$ is as above.  But note that the second sum may be written as
\begin{multline*}
\sum_{\substack{E \subseteq \{1, \dots, n-1\} \setminus \{i_{1}, \dots, i_{k} \}}} \prod_{j \in E} \frac{(-t)^{n-j}}{1-(-t)^{n-j}} = \prod_{j \in \{1, \dots, n-1 \} \setminus \{i_{1}, \dots, i_{k} \}} \Bigg(1 + \frac{(-t)^{n-j}}{1-(-t)^{n-j}} \Bigg) \\
= \prod_{j \in \{1, \dots, n-1 \} \setminus \{i_{1}, \dots, i_{k} \}} \Bigg( \frac{1}{1-(-t)^{n-j}} \Bigg).
\end{multline*}
So we have
\begin{multline*}
\lim_{c \rightarrow -\infty} ch_{s}(M_{c}(1^{n})) \\ = \frac{1}{1+ (-t)^{n}} \sum_{\sigma \in S(n)} (-1)^{k} \prod_{l=1}^{k} \frac{(-t)^{n-i_{l}}}{1-(-t)^{n-i_{l}}} \prod_{j \in \{1, \dots, n-1 \} \setminus \{i_{1}, \dots, i_{k} \}} \Bigg( \frac{1}{1-(-t)^{n-j}} \Bigg) \\
= \Bigg(\frac{1}{1+ (-t)^{n}} \prod_{i=1}^{n-1} \frac{1}{1-(-t)^{n-i}} \Bigg) \sum_{\sigma \in S(n)} \Bigg((-1)^{k} \prod_{l=1}^{k} (-t)^{n-i_{l}}\Bigg).
\end{multline*}
\end{proof}

\begin{remarks}
We can rewrite (\ref{signptnsimp}) as 
\begin{multline}\label{multisimp}
\lim_{c \rightarrow -\infty} ch_{s}(M_{c}(1^{n})) \\ = \Bigg(\frac{1}{1+ (-t)^{n}} \prod_{i=1}^{n-1} \frac{1}{1-(-t)^{n-i}} \Bigg) \sum_{\substack{1 \leq i_{1} < \cdots < i_{k} \leq n-1\\ k \geq 0}} \Bigg(C_{(i_{1}, \dots, i_{k})}(-1)^{k} \prod_{l = 1}^{k} (-t)^{n-i_{l}}\Bigg),
\end{multline}
where 
\begin{multline*}
C_{(i_{1}, \dots, i_{k})} = |\{\sigma \in S(n) : \mathcal{I}(\sigma) = \{i_{1}, \dots, i_{k} \} \}|  \\
= \sum_{S \subseteq \{i_{1}, \dots, i_{k} \}} (-1)^{k - |S|} \binom{n}{n-j_{s}, j_{s}-j_{s-1}, \dots, j_{2}-j_{1}, j_{1}}
\end{multline*}
with $S = \{ j_{1}, \dots, j_{s} \}$ (and if $S = \emptyset$, the multinomial coefficient is one).  Note that $C_{(i_{1}, \dots, i_{k})}$ is the number of permutations with descent set equal to $\{i_{1}, \dots, i_{k} \}$.
\end{remarks}

Using arguments analgous to the $\lambda = (1^{n})$ case, we obtain for arbitary $\lambda$:
\begin{theorem} Let $|\lambda| = n$.  We have the asymptotic limit
\begin{multline} \label{asymsimpl}
\lim_{c \rightarrow -\infty} ch_{s}(M_{c}(\lambda)) \\
= \sum_{\substack{T \in Std(\lambda) \\ \text{with content } (d_{1}, \dots, d_{n})}} \Big( \prod_{j=0}^{n-1} \frac{1}{1-(-1)^{k_{j}(T)}t^{n-j}} \Big) \sum_{\sigma \in S(n)} \Big((-1)^{|\mathcal{J}_{T}(\sigma)|} \prod_{i \in \mathcal{I}(\sigma)} (-1)^{k_{i}(T)} t^{n-i} \Big)
\end{multline}
where $\mathcal{J}_{T}(\sigma) = \{\text{inversions $(s,t)$ of $\sigma$ which satisfy } d_{t}-d_{s} = 0,-1 \}$, $\mathcal{I}(\sigma) = \{ 1 \leq i \leq n-1: \sigma(i)> \sigma(i+1)\}$ and
\begin{equation*}
k_{j}(T) = \# \{ i: d_{i} < 0 \text{ and } i>j \} + \# \{ (s,t) : 0 \neq s \leq j<t \text{ and } d_{t}-d_{s} = 0,-1 \}.
\end{equation*}
\end{theorem}

We will now provide a formula for $\lambda = (1^{n})$ in terms of specializations of certain polynomials.  Let $r \leq n \in \mathbb{Z}_{+}$.  Let $\mathcal{P}_{n}(x_{1}, \dots, x_{r})$ denote the sum of monomials in the expansion of $(x_{1} + \cdots + x_{r})^{n}$  with exponent on each $x_{i}$ for $1 \leq i \leq r$ positive, i.e.:
\begin{equation} \label{Ppol}
\mathcal{P}_{n}(x_{1}, \dots,  x_{r}) := \sum_{\substack{k_{1} + \cdots + k_{r} = n \\ k_{i} > 0}} \binom{n}{k_{1}, \dots, k_{r}} x_{1}^{k_{1}} \cdots x_{r}^{k_{r}}.
\end{equation}  
Note that $\mathcal{P}_{n}(x_{1}, \dots, x_{r})$ can be computed in terms of $(x_{i_{1}} + \cdots + x_{i_{k}})^{n}$ from the multinomial formula using Inclusion-Exclusion:
\begin{equation*}
\mathcal{P}_{n}(x_{1}, \dots, x_{r}) = \sum_{\emptyset \neq S \subseteq \{1, \dotsc, r\}} (-1)^{r - |S|} \Big(\sum_{i \in S} x_{i}\Big)^{n}.
\end{equation*}
For example, for $n \geq 2$, $\mathcal{P}_{n}(x_{1}) = x_{1}^{n}$ and 
\begin{equation*}
\mathcal{P}_{n}(x_{1}, x_{2}) = \sum_{0<r<n} \binom{n}{r,n-r} x_{1}^{r}x_{2}^{n-r} = (x_{1} + x_{2})^{n} - x_{1}^{n} - x_{2}^{n}. 
\end{equation*}

\begin{theorem} \label{sgnasym}
Let $n$ be fixed.  Then we have
\begin{equation*}
\lim_{c \rightarrow -\infty} ch_{s}(M_{c}(1^{n})) =  \Bigg(\frac{1}{1+ t^{n}} \prod_{j=1}^{n-1} \frac{1+t^{j}}{1-t^{j}} \Bigg) 
 \Bigg[ 1 + \sum_{s=1}^{n-1} \sum_{\alpha \in \mathcal{I}_{s}} (-1)^{\alpha_{s}} \mathcal{P}_{n}(1, t^{\alpha_{1}}, t^{\alpha_{2}}, \dots, t^{\alpha_{s}})\Bigg],
\end{equation*}
where $\mathcal{I}_{s}$ is the set of strictly increasing sequences of length $s$, i.e.,
\begin{equation*}
\mathcal{I}_{s} = \{(\alpha_{1}, \dots, \alpha_{s}) : 0< \alpha_{1} < \alpha_{2} < \cdots < \alpha_{s} \}
\end{equation*}
and the polynomial $\mathcal{P}_{n}(x_{1}, \dots, x_{r})$ for any $r \leq n$ is defined in (\ref{Ppol}).
\end{theorem}
\begin{proof}

Using Equation (\ref{multisimp}), (and replacing $-t$ by $t$) we have
\begin{multline*}
\lim_{c \rightarrow \infty} ch_{s}(M_{c}(n)) = \Bigg(\frac{1}{1+ t^{n}} \prod_{i=1}^{n-1} \frac{1}{1-t^{n-i}} \Bigg) \times  \\ \sum_{\substack{1 \leq i_{1}< \cdots < i_{k} \leq n-1 \\ k \geq 0}} \sum_{S \subset \{ i_{1}, \dots, i_{k} \}} (-1)^{s} \binom{n}{n-j_{s}, j_{s} - j_{s-1}, \dots, j_{2}-j_{1}, j_{1}} t^{\sum_{l=1}^{k}(n-i_{l})},
\end{multline*}
where $S = \{j_{1}, \dots, j_{s} \}$.  Switching the order of summation, the sum is equal to
\begin{multline*}
\sum_{S = \{j_{1}, \dots, j_{s} \}}  (-1)^{s} \binom{n}{n-j_{s}, j_{s} - j_{s-1}, \dots, j_{2}-j_{1}, j_{1}} \sum_{S \subseteq I} t^{\sum_{i \in I}(n-i)}.
\end{multline*}
We compute the inner sum:
\begin{multline*}
\sum_{S \subseteq I} t^{\sum_{i \in I}(n-i)} = t^{\sum_{s \in S} (n-s)} \sum_{I \subseteq \{1, \dots, n-1 \} \setminus S } \Bigg( \prod_{i \in I} t^{n-i} \Bigg) = t^{\sum_{s \in S} (n-s)}  \prod_{i \in \{1, \dots, n-1\} \setminus S} (1+t^{n-i}).
\end{multline*}
So putting this back into the original equation, we have
\begin{multline} \label{full}
\lim_{c \rightarrow -\infty} ch_{s}(M_{c}(1^{n}) = \Bigg(\frac{1}{1+ t^{n}} \prod_{i=1}^{n-1} \frac{1+t^{n-i}}{1-t^{n-i}} \Bigg) \times \\
 \sum_{\substack{S = \{j_{1}, \dots, j_{s}\} \\ s \geq 0}} (-1)^{s} \prod_{i=1}^{s} \frac{t^{n-j_{i}}}{1+t^{n-j_{i}}} \binom{n}{n-j_{s}, j_{s} - j_{s-1}, \dots, j_{2}-j_{1}, j_{1}},
\end{multline}
with $1 \leq j_{1} < \cdots < j_{s} \leq n-1$.
We fix $0 \leq s \leq n-1$ and we want to compute the inner sum.  We reindex, using $l_{1} = j_{1} -0, l_{2} = j_{2} - j_{1}, \dots, l_{s} = j_{s} - j_{s-1}, l_{s+1}  = n-j_{s} = n - \sum_{i=1}^{s} l_{i}$.  Note $l_{i} > 0$.  The sum above (for fixed $s$) becomes
\begin{equation} \label{inner}
\sum_{\substack{l_{i} \in \mathbb{Z}_{>0} \\ i = 1, \dots, s \\ \sum_{i} l_{i} < n}} \prod_{i=1}^{s} \frac{t^{n - \sum_{j=1}^{i} l_{j}}}{1 + t^{n - \sum_{j=1}^{i} l_{j}}} \binom{n}{l_{1},l_{2}, \dots, l_{s}, n- \sum_{j=1}^{s} l_{j}}.
\end{equation}
We expand the rational functions inside the sum and rewrite (\ref{inner}) as
\begin{multline*}
\sum_{\substack{l_{i} \in \mathbb{Z}_{>0} \\ i = 1, \dots, s \\ \sum_{i} l_{i} < n}} \sum_{\substack{\lambda_{i} \in \mathbb{Z}_{> 0} \\ i=1, \dots, s}}\prod_{i=1}^{s} (-1)^{\lambda_{i}-1}t^{\lambda_{i}(n - \sum_{j=1}^{i} l_{j})} \binom{n}{l_{1},l_{2}, \dots, l_{s}, n- \sum_{j=1}^{s} l_{j}} \\
=\sum_{\substack{\lambda_{i} \in \mathbb{Z}_{> 0} \\ i=1, \dots, s}} (t^{\sum_{i} \lambda_{i}})^{n} (-1)^{\sum_{i} \lambda_{i}}(-1)^{s} \sum_{\substack{l_{i} \in \mathbb{Z}_{>0} \\ i = 1, \dots, s \\ \sum_{i} l_{i} < n}} \prod_{i=1}^{s} t^{-(\lambda_{i} + \cdots + \lambda_{s})l_{i}}\binom{n}{l_{1},l_{2}, \dots, l_{s}, n- \sum_{j=1}^{s} l_{j}}.
\end{multline*}
Note that the inner sum can be computed by Inclusion-Exclusion and the multinomial formula.  Namely, the inner sum is equal to the sum of terms with each exponent greater than zero in the expansion of
\begin{equation*}
(t^{-(\lambda_{1} + \cdots + \lambda_{s})} + t^{-(\lambda_{2} + \cdots + \lambda_{s})} + \cdots + t^{-\lambda_{s}} + 1)^{n}.
\end{equation*}
We multiply by the factor $(t^{\sum_{i} \lambda_{i}})^{n}$ outside the inner sum
\begin{multline*}
(t^{\sum_{i} \lambda_{i}})^{n}(t^{-(\lambda_{1} + \cdots + \lambda_{s})} + t^{-(\lambda_{2} + \cdots + \lambda_{s})} + \cdots + t^{-\lambda_{s}} + 1)^{n} = (1 + t^{\lambda_{1}} + t^{\lambda_{1} + \lambda_{2}} + \cdots + t^{\lambda_{1} + \cdots + \lambda_{s}})^{n}
\end{multline*}
Note that $\lambda_{1} + \cdots + \lambda_{s} > \lambda_{1} + \cdots + \lambda_{s-1} > \cdots > \lambda_{1} > 0$ is a \textit{strict} partition.  We reparametrize via:
\begin{equation*}
(\lambda_{1} + \cdots + \lambda_{n-1}, \lambda_{1} + \cdots + \lambda_{n-2}, \dots, \lambda_{1}) = (\beta_{1}, \dots, \beta_{n-1}) + (n-1,n-2, \dots, 1)
\end{equation*}
for $\beta \in \mathcal{P}_{+}^{n-1}$ a partition ($= \{ \lambda_{1} \geq \lambda_{2} \geq \cdots \geq \lambda_{n-1} \geq 0\}$).  Also we have
\begin{equation*}
\lambda_{1} + \cdots + \lambda_{s} = \beta_{n-s} + s.
\end{equation*}
So summing over $s$ and accounting for $(-1)^{s}$ (without rational factors in (\ref{full}) out front) we get
\begin{equation*}
1 + \sum_{s=1}^{n-1}  \sum_{\substack{\beta \in \mathcal{P}_{+}^{s}\\ = (\beta_{1}, \dots, \beta_{s})}}(-1)^{\beta_{1} + s} \mathcal{P}_{n}(1 , t^{\beta_{s} + 1}, t^{\beta_{s-1} + 2} , \dots , t^{\beta_{1} + s}).
\end{equation*}
Thus, we get
\begin{multline*}
\lim_{c \rightarrow -\infty} ch_{s}(M_{c}(1^{n})) = \Bigg(\frac{1}{1+ t^{n}} \prod_{i=1}^{n-1} \frac{1+t^{n-i}}{1-t^{n-i}} \Bigg) \\
\times \Bigg[ 1 + \sum_{s=1}^{n-1}  \sum_{\substack{\beta \in \mathcal{P}_{+}^{s}}}(-1)^{\beta_{1} + s} \mathcal{P}_{n}(1 , t^{\beta_{s} + 1}, t^{\beta_{s-1} + 2} , \dots , t^{\beta_{1} + s}) \Bigg].
\end{multline*}
Finally, we reindex by $(\beta_{s} + 1, \beta_{s-1} + 2, \cdots , \beta_{1} + s) = (\alpha_{1}, \dots, \alpha_{s})$.
\end{proof}

We use Theorem \ref{sgnasym} to compute some examples of 
\begin{equation*}
\lim_{c \rightarrow -\infty} ch_{s}(M_{c}(1^{n})),
\end{equation*}
for small values of $n$.

\textbf{Example.} For $n=2$, we only have $s=1$ and by Equation (\ref{Ppol}), $\mathcal{P}_{2}(x_{1}, x_{2}) = 2x_{1}x_{2}$, so we get
\begin{equation*}
1 + \sum_{r = 0}^{\infty} (-1)^{r + 1} \mathcal{P}_{2}(1, t^{r+1}) = 1 + 2\sum_{r=0}^{\infty} (-1)^{r+1} t^{r+1} = 1 -2t \frac{1}{1+t} = \frac{1-t}{1+t}.
\end{equation*}
The prefactors are
\begin{equation*}
\frac{1}{1+t^{2}}\frac{1+t}{1-t},
\end{equation*}
so 
\begin{equation*}
\lim_{c \rightarrow -\infty} ch_{s}(M_{c}(1^{2})) = \frac{1}{1+t^{2}}.
\end{equation*}

\textbf{Example.} For $n=3$, we have terms corresponding to $s=1$ and $s=2$ and by Equation (\ref{Ppol}), we compute $\mathcal{P}_{3}(x_{1}, x_{2}) = 3x_{1}^{2}x_{2} + 3x_{1}x_{2}^{2}$ and $\mathcal{P}_{3}(x_{1} , x_{2} , x_{3}) = 6x_{1}x_{2}x_{3}$, so we get
\begin{multline*}
1 + \sum_{\beta_{1} \geq 0} (-1)^{\beta_{1} + 1} \mathcal{P}_{3}(1, t^{\beta_{1}+1}) + \sum_{\beta_{1} \geq \beta_{2} \geq 0}(-1)^{\beta_{1}+2} \mathcal{P}_{3}(1, t^{\beta_{2}+ 1} , t^{\beta_{1}+2}) \\
= 1 + \sum_{\beta_{1} \geq 0} (-1)^{\beta_{1} + 1} (3t^{\beta_{1} + 1} + 3t^{2\beta_{1} + 2}) + \sum_{\beta_{1} \geq \beta_{2} \geq 0}(-1)^{\beta_{1} + 2}6t^{\beta_{2} + 1}t^{\beta_{1} + 2} \\
= 1 + \sum_{\beta_{1} \geq 0}\Big(3(-1)^{\beta_{1} + 1}t^{\beta_{1} + 1} + 3(-1)^{\beta_{1} + 1}t^{2\beta_{1} + 2}\Big) + \sum_{d_{1},d_{2} \geq 0}6(-1)^{d_{1}+d_{2} + 2}t^{d_{2}+1}t^{d_{1}+d_{2}+2} \\
= 1 -3t\sum_{\beta_{1} \geq 0} (-t)^{\beta_{1}} - 3t^{2}\sum_{\beta_{1} \geq 0} (-1)^{\beta_{1}}t^{2\beta_{1}} + 6t^{3} \sum_{d_{1}, d_{2} \geq 0} (-1)^{d_{1} + d_{2}} t^{d_{2}} t^{d_{1}+d_{2}} \\
= 1-\frac{3t}{1+t} - \frac{3t^{2}}{1 + t^{2}} + \frac{6t^{3}}{(1+t)(1+t^{2})} = \frac{(1+t)(1+t^{2}) - 3t(1+t^{2}) - 3t^{2}(1+t) + 6t^{3}}{(1+t)(1+t^{2})} \\
= \frac{1 + t + t^{2} + t^{3} - 3t - 3t^{3} - 3t^{2} - 3t^{3} + 6t^{3}}{(1+t)(1+t^{2})} = \frac{1-2t - 2t^{2} + t^{3}}{(1+t)(1+t^{2})}
\end{multline*}
where we have reparametrized by $d_{1} = \beta_{1} - \beta_{2}$ and $d_{2} = \beta_{2} - 0$, so $\beta_{2} = d_{2}$ and $\beta_{1} = d_{1} + d_{2}$.
The prefactors are
\begin{equation*}
\frac{1}{1+t^{3}} \frac{1+t}{1-t} \frac{1+t^{2}}{1-t^{2}},
\end{equation*}
so
\begin{equation*}
\lim_{c \rightarrow -\infty} ch_{s}(M_{c}(1^{3})) = \frac{1-2t - 2t^{2} + t^{3}}{(1+t)(1+t^{2})} \times \frac{1}{1+t^{3}} \frac{1+t}{1-t} \frac{1+t^{2}}{1-t^{2}} = \frac{1-2t-2t^{2} + t^{3}}{(1+t^{3})(1-t)(1-t^{2})}. 
\end{equation*}

\textbf{Example.} For $n=4$, we have terms corresponding to $s=1,2,3$ and by Equation (\ref{Ppol}), we compute 
\begin{equation*}
\mathcal{P}_{4}(x_{1}, x_{2}) = 6x_{1}^{2}x_{2}^{2} + 4x_{1}^{3}x_{2} + 4x_{1}x_{2}^{3} 
\end{equation*}
\begin{equation*}
\mathcal{P}_{4}(x_{1}, x_{2}, x_{3}) = 12x_{1}^{2}x_{2}x_{3} + 12x_{1}x_{2}^{2}x_{3} + 12x_{1}x_{2}x_{3}^{2}
\end{equation*}
\begin{equation*}
\mathcal{P}_{4}(x_{1}, x_{2}, x_{3}, x_{4}) = 24x_{1}x_{2}x_{3}x_{4}.
\end{equation*}

We use this to compute
\begin{equation*}
\sum_{\alpha_{1} > 0} (-1)^{\alpha_{1}} \mathcal{P}_{4}(1, t^{\alpha_{1}}) = \sum_{\alpha_{1} > 0} \Bigg[ (-1)^{\alpha_{1}} 6t^{2\alpha_{1}} + (-1)^{\alpha_{1}} 4t^{\alpha_{1}} + (-1)^{\alpha_{1}} 4t^{3\alpha_{1}} \Bigg]
\end{equation*}

Note that
\begin{equation*}
\sum_{n >0} (-t)^{n} = \sum_{n \geq 0} (-t)^{n} - 1 = \frac{-t}{1+t},
\end{equation*}
so the above equation is equal to 
\begin{equation*}
\frac{-6t^{2}}{1+t^{2}} - \frac{4t}{1+t} - \frac{4t^{3}}{1+t^{3}}.
\end{equation*}

Similarly, we compute
\begin{multline*}
\sum_{0< \alpha_{1} < \alpha_{2}} (-1)^{\alpha_{2}} \mathcal{P}_{4}(1, t^{\alpha_{1}}, t^{\alpha_{2}}) = \sum_{d_{1}, d_{2} > 0} (-1)^{d_{1} + d_{2}} \mathcal{P}_{4}(1, t^{d_{1}}, t^{d_{1} + d_{2}}) \\= \sum_{d_{1}, d_{2} > 0} (-1)^{d_{1} + d_{2}} (12t^{d_{1}}t^{d_{1}+d_{2}} + 12t^{2d_{1}}t^{d_{1}+d_{2}} + 12t^{d_{1}}t^{2(d_{1}+d_{2})}) \\
= \frac{12t^{2}t}{(1+t^{2})(1+t)} + \frac{12t^{3}t}{(1+t^{3})(1+t)} + \frac{12t^{3}t^{2}}{(1+t^{3})(1+t^{2})}.
\end{multline*}
and
\begin{multline*}
\sum_{0 < \alpha_{1} < \alpha_{2} < \alpha_{3}} (-1)^{\alpha_{3}} \mathcal{P}_{4}(1, t^{\alpha_{1}}, t^{\alpha_{2}}, t^{\alpha_{3}} ) = \sum_{d_{1},d_{2},d_{3} > 0} (-1)^{d_{1}+d_{2}+d_{3}} \mathcal{P}_{4}(1, t^{d_{1}}, t^{d_{1}+d_{2}}, t^{d_{1}+d_{2}+d_{3}})
\\ = \sum_{d_{1},d_{2},d_{3} > 0} (-1)^{d_{1}+d_{2}+d_{3}} 24 t^{d_{1}}t^{d_{1}+d_{2}}t^{d_{1}+d_{2}+d_{3}} = \frac{-24t^{3}t^{2}t}{(1+t^{3})(1+t^{2})(1+t)}
\end{multline*}

Adding this together, adding one and multiplying by the prefactors, we obtain
\begin{multline*}
\lim_{c \rightarrow -\infty} ch_{s}(M_{c}(1^{4})) =  \Bigg( \frac{1}{1+t^{4}} \frac{(1+t)(1+t^{2})(1+t^{3})}{(1-t)(1-t^{2})(1-t^{3})} \Bigg) \times \\ \Bigg[ 1 - \frac{6t^{2}}{1+t^{2}}  - \frac{4t}{1+t} - \frac{4t^{3}}{1+t^{3}} + \frac{12t^{3}}{(1+t)(1+t^{2})} + \frac{12t^{4}}{(1+t)(1+t^{3})} + \\  \frac{12t^{5}}{(1+t^{2})(1+t^{3})} - \frac{24t^{6}}{(1+t)(1+t^{2})(1+t^{3})}\Bigg]. 
\end{multline*}

Finally, we use the previous theorem to give a formula for this limiting case in terms of explicit rational functions in $t$.

\begin{theorem} \label{sgnasym2}
Let $n$ be fixed.  We have
\begin{multline*}
\lim_{c \rightarrow -\infty} ch_{s}(M_{c}(1^{n})) = \Bigg(\frac{1}{1+ t^{n}} \prod_{j=1}^{n-1} \frac{1+t^{j}}{1-t^{j}} \Bigg) \times \\
\times
 \Bigg[ \sum_{\substack{0 \leq s \leq n-1 \\ k_{0} + k_{1} + \cdots + k_{s}=n \\ k_{0}, \dots, k_{s} > 0}}  \binom{n}{k_{0}, k_{1}, \dots, k_{s}}  \frac{(-1)^{s} t^{k_{1} + 2k_{2} + \cdots + sk_{s}} }{(1+t^{k_{1} +k_{2} + \cdots + k_{s}})(1+t^{k_{2} + \cdots + k_{s}}) \cdots (1+t^{k_{s}})} \Bigg].
\end{multline*}
\end{theorem}
\begin{proof}
Terms in Theorem \ref{sgnasym} are of the form (varying over $s$ with $2 \leq s+1 \leq n$)
\begin{equation*}
\sum_{d_{1}, d_{2}, \cdots, d_{s} >0} (-1)^{d_{1} + d_{2} + \cdots + d_{s}} \mathcal{P}_{n}(1, t^{d_{1}}, t^{d_{1}+d_{2}}, \cdots, t^{d_{1} + d_{2} + \cdots + d_{s}}).
\end{equation*}
Recall $\mathcal{P}_{n}(1, x_{1}, \dots, x_{s})$ is a sum with an arbitrary term of the form
\begin{equation*}
\binom{n}{k_{0}, k_{1}, \dots, k_{s}} x_{1}^{k_{1}} \cdots x_{s}^{k_{s}},
\end{equation*}
where $k_{0} + \cdots + k_{s} = n$ and $k_{i} > 0$.  So an arbitrary term of $$(-1)^{d_{1} + \cdots + d_{s}}\mathcal{P}_{n}(1, t^{d_{1}}, t^{d_{1}+d_{2}}, \cdots, t^{d_{1} + d_{2} + \cdots + d_{s}})$$ is of the form
\begin{multline*}
(-1)^{d_{1} + \cdots + d_{s}}\binom{n}{k_{0}, k_{1}, \dots, k_{s}} t^{d_{1}k_{1}} t^{k_{2}(d_{1}+d_{2})} \cdots t^{k_{s}(d_{1}+d_{2} + \cdots + d_{s})}  \\ = \binom{n}{k_{0}, k_{1}, \dots, k_{s}}  (-1)^{d_{1} + \cdots + d_{s}} t^{d_{1}(k_{1} + k_{2} + \cdots + k_{s})} t^{d_{2}(k_{2} + \cdots + k_{s})} \cdots t^{d_{s}k_{s}},
\end{multline*}
and summing over $d_{1}, \dots, d_{s} > 0$ gives the following rational function in $t$
\begin{multline*}
\binom{n}{k_{0}, k_{1}, \dots, k_{s}} \sum_{d_{1}, d_{2}, \dots, d_{s} > 0} (-1)^{d_{1} + \cdots + d_{s}} t^{d_{1}(k_{1} + k_{2} + \cdots + k_{s})} t^{d_{2}(k_{2} + \cdots + k_{s})} \cdots t^{d_{s}k_{s}} \\
= \binom{n}{k_{0}, k_{1}, \dots, k_{s}}  \frac{(-1)^{s} t^{k_{1} + k_{2} + \cdots + k_{s}} t^{k_{2} + \cdots + k_{s}} \cdots t^{k_{s}} }{(1+t^{k_{1} +k_{2} + \cdots + k_{s}})(1+t^{k_{2} + \cdots + k_{s}}) \cdots (1+t^{k_{s}})} \\ = \binom{n}{k_{0}, k_{1}, \dots, k_{s}}  \frac{(-1)^{s} t^{k_{1} + 2k_{2} + \cdots + sk_{s}} }{(1+t^{k_{1} +k_{2} + \cdots + k_{s}})(1+t^{k_{2} + \cdots + k_{s}}) \cdots (1+t^{k_{s}})}.
\end{multline*}
Summing over $0 \leq s \leq n-1$ (with $s=0$ giving the term $1$ outside the sum) gives the result.
\end{proof}

Note that taking the power series expansion of the formula in Theorem \ref{sgnasym2}, we have
\begin{equation*}
\lim_{c \rightarrow -\infty} ch_{s}(M_{c}(1^{n})) = \sum_{r \geq 0} g_{r}(n) t^{r},
\end{equation*}
where $g_{r}(n)$ is equal to a polynomial $P_{r}(n)$ for $n \geq r+1$.  Indeed, the series coefficients of the inner sum within the parantheses are a finite linear combination of multinomial coefficients in $n$, and for $n \geq r+1$ the series coefficients of the prefactor on $t^{r}$ are constant.  We have the ``stable limit"
\begin{multline*}
f(a,t) := \sum_{r \geq 0} P_{r}(a) t^{r} = \prod_{j=1}^{\infty} \frac{1+t^{j}}{1-t^{j}}  \\ \cdot \negthickspace \negthickspace \negthickspace \sum_{\substack{s \geq 0 \\k_{1}, \dots, k_{s} > 0}} \negthickspace \negthickspace \negthickspace \frac{a(a-1) \cdots (a-(k_{1} + \cdots + k_{s}-1))}{k_{1}! \cdots k_{s}!} \frac{(-1)^{s}t^{k_{1}+2k_{2} + \cdots + sk_{s}}}{(1+t^{k_{1}+k_{2} + \cdots + k_{s}})(1+t^{k_{2} + \cdots + k_{s}}) \cdots (1+t^{k_{s}})}.
\end{multline*}

The first few values of $P_{r}(n)$ are as follows (computed in SAGE):
\begin{align*}
P_{0}(n) &= 1 \\
P_{1}(n) &= 2-n \\
P_{2}(n) &= 4 - \frac{1}{2}n - \frac{1}{2}n^{2} \\
P_{3}(n) &= 8 - \frac{10}{3}n + \frac{1}{2}n^{2} - \frac{1}{6}n^{3} \\ 
P_{4}(n) &= 14 - \frac{47}{12}n - \frac{35}{24}n^{2} + \frac{5}{12}n^{3} - \frac{1}{24} n^{4}.
\end{align*} 

We will now provide an explicit formula for the stable limit $f(a,t)$.

\begin{theorem} \label{stablimit}
We have 
\begin{equation*}
f(a,t) =  \prod_{j=1}^{\infty} \frac{1+t^{j}}{1-t^{j}}  \times \Big(1 + \sum_{i \geq 1} (-1)^{i} \Big( [i+1]^{a} - [i]^{a} \Big) \Big),
\end{equation*}
where 
\begin{equation*}
[i] = \frac{1-t^{i}}{1-t}.
\end{equation*}
\end{theorem}
\begin{proof}
We first recall the following formula for the $\Gamma$-function:
\begin{equation} \label{gamma}
\Gamma(a)^{-1} \int_{0}^{\infty} x^{a+m-1}e^{-x} dx = a(a+1) \cdots (a+m-1).
\end{equation}
From the remarks following the previous theorem, we have the following formula for the stable limit
\begin{multline*}
f(-a,t) = \prod_{j=1}^{\infty} \frac{1+t^{j}}{1-t^{j}}  \sum_{\substack{s \geq 0 \\k_{1}, \dots, k_{s} > 0}} \Bigg[ \frac{-a(-a-1) \cdots (-a-(k_{1} + \cdots + k_{s} - 1))}{k_{1}! \cdots k_{s}!} \times \\ \frac{(-1)^{s}t^{k_{1}+2k_{2} + \cdots + sk_{s}}}{(1+t^{k_{1}+k_{2} + \cdots + k_{s}})(1+t^{k_{2} + \cdots + k_{s}}) \cdots (1+t^{k_{s}})} \Bigg]\\
=  \prod_{j=1}^{\infty} \frac{1+t^{j}}{1-t^{j}} \times \sum_{\substack{s \geq 0 \\k_{1}, \dots, k_{s} > 0}} \Bigg[ \frac{(-1)^{k_{1} + \cdots + k_{s}}a(a+1) \cdots (a+(k_{1} + \cdots + k_{s} - 1))}{k_{1}! \cdots k_{s}!}\times \\ \frac{(-1)^{s}t^{k_{1}+2k_{2} + \cdots + sk_{s}}}{(1+t^{k_{1}+k_{2} + \cdots + k_{s}})(1+t^{k_{2} + \cdots + k_{s}}) \cdots (1+t^{k_{s}})} \Bigg]\\
= \prod_{j=1}^{\infty} \frac{1+t^{j}}{1-t^{j}} \sum_{\substack{s \geq 0 \\k_{1}, \dots, k_{s} > 0}} \Bigg[ \frac{(-1)^{k_{1} + \cdots + k_{s}}\frac{1}{\Gamma(a)}\int_{0}^{\infty} x^{a + k_{1} + \cdots + k_{s}-1}e^{-x}dx}{k_{1}! \cdots k_{s}!}\times \\ \frac{(-1)^{s}t^{k_{1}+2k_{2} + \cdots + sk_{s}}}{(1+t^{k_{1}+k_{2} + \cdots + k_{s}})(1+t^{k_{2} + \cdots + k_{s}}) \cdots (1+t^{k_{s}})} \Bigg] \\
=  \prod_{j=1}^{\infty} \frac{1+t^{j}}{1-t^{j}} \sum_{\substack{s \geq 0 \\k_{1}, \dots, k_{s} > 0}} \Bigg[\frac{\frac{1}{\Gamma(a)}\int_{0}^{\infty} (-x)^{ k_{1} + \cdots + k_{s}}x^{a-1}e^{-x}dx}{k_{1}! \cdots k_{s}!}\times \\ \frac{(-1)^{s}t^{k_{1}+2k_{2} + \cdots + sk_{s}}}{(1+t^{k_{1}+k_{2} + \cdots + k_{s}})(1+t^{k_{2} + \cdots + k_{s}}) \cdots (1+t^{k_{s}})}\Bigg].
\end{multline*}
We will first simplify
\begin{equation*}
h(x,t) := \sum_{\substack{s \geq 0 \\ k_{1}, \dots, k_{s} > 0}} \frac{(-1)^{s}t^{k_{1}+2k_{2} + \cdots + sk_{s}}x^{k_{1} + \dots + k_{s}}}{(1+t^{k_{1}+k_{2} + \cdots + k_{s}})(1+t^{k_{2} + \cdots + k_{s}}) \cdots (1+t^{k_{s}})k_{1}! \cdots k_{s}!}.
\end{equation*}
Note that
\begin{equation*}
f(-a,t) =  \prod_{j=1}^{\infty} \frac{1+t^{j}}{1-t^{j}}  \times \frac{1}{\Gamma(a)} \int_{0}^{\infty} x^{a-1}e^{-x}h(-x,t) \;dx.
\end{equation*}
Writing
\begin{equation*}
\frac{t^{k_{i} + \cdots + k_{s}}}{1+t^{k_{i} + \cdots + k_{s}}} = \sum_{r_{i} > 0} (-1)^{r_{i}-1} t^{r_{i}(k_{i} + \cdots + k_{s})},
\end{equation*}
we obtain
\begin{equation*}
h(x,t) = \sum_{s \geq 0} \sum_{k_{1}, \dots, k_{s} >0} \sum_{r_{1}, \dots, r_{s} >0} \frac{(-1)^{s}(-1)^{r_{1} + \cdots + r_{s} -s} t^{r_{1}(k_{1} + \cdots + k_{s})  + \cdots + r_{s}k_{s}}x^{k_{1} + \cdots + k_{s}}}{k_{1}! \cdots k_{s}!}.
\end{equation*}
Summing over the $k_{i}$ on the inside allows us to rewrite this in terms of exponential functions as
\begin{equation*}
h(x,t) = \sum_{s \geq 0} \sum_{r_{1}, \dots, r_{s}>0} (-1)^{r_{1} + \cdots + r_{s}} (e^{xt^{r_{1} + \cdots + r_{s}}}-1 ) \cdots(e^{xt^{r_{1} + r_{2}}}-1) (e^{xt^{r_{1}}}-1).
\end{equation*}
Now let $r_{1} = b_{1}$, $r_{1} + r_{2} = b_{2}$,...,$r_{1} + \cdots + r_{s} = b_{s}$.  Note that $b_{1} < b_{2} < \cdots < b_{s}$ is an increasing sequence of positive integers.  Thus, with this notation, the above formula becomes
\begin{multline*}
h(x,t) = \sum_{s \geq 0} \sum_{0 < b_{1} < \cdots < b_{s}} (-1)^{b_{s}} (e^{xt^{b_{s}}} - 1) \cdots (e^{xt^{b_{1}}}-1) \\
= 1 + \sum_{s > 0} \sum_{0 < b_{1} < \cdots < b_{s}} (-1)^{b_{s}} (e^{xt^{b_{s}}} - 1) \cdots (e^{xt^{b_{1}}}-1).
\end{multline*}
We now let $z_{i} = (e^{xt^{i}}-1)$.  Then the double sum above is enumerated by 
\begin{equation*}
\sum_{i \geq 1} (-1)^{i} z_{i}(1+z_{1}) \cdots (1+z_{i-1});
\end{equation*}
indeed, expanding out the parantheses in the latter sum shows that the two sets of of terms are identical.  Thus,
\begin{multline*}
h(x,t) = 1 + \sum_{i \geq 1} (-1)^{i} z_{i}(1+z_{1}) \cdots (1+z_{i-1}) 
= 1 + \sum_{i \geq 1} (-1)^{i} (e^{xt^{i}}-1)e^{x(t + \cdots + t^{i-1})} \\ = 1 + \sum_{i \geq 1} (-1)^{i} \Big( e^{xt[i]} - e^{xt[i-1]} \Big).
\end{multline*}
Finally, substituting this back in to the equation above we have
\begin{equation*}
f(-a,t) = \prod_{j=1}^{\infty} \frac{1+t^{j}}{1-t^{j}} \times  \frac{1}{\Gamma(a)} \int_{0}^{\infty} x^{a-1}e^{-x} \Big( 1 + \sum_{i \geq 1} (-1)^{i}\big( e^{-xt[i]} - e^{-xt[i-1]} \big) \Big)dx.
\end{equation*}
We first compute one of these integrals in the sum
\begin{equation*}
\int_{0}^{\infty} x^{a-1} e^{-x}e^{-xt[j]} = \int_{0}^{\infty} x^{a-1} e^{-x[j+1]}
\end{equation*}
since $t[j] + 1 = \frac{t - t^{j+1} + 1-t}{1-t} = \frac{1-t^{j+1}}{1-t} = [j+1]$.
Using $u$-substitution and Equation (\ref{gamma}) with $m=1$, one can check that this integral is equal to 
\begin{equation*}
\Gamma(a) \Big( \frac{1-t}{1-t^{j+1}} \Big)^{a}.
\end{equation*}
Thus, we obtain
\begin{equation*}
f(-a,t) = \prod_{j=1}^{\infty} \frac{1+t^{j}}{1-t^{j}} \times \Big( 1 + \sum_{i \geq 1} (-1)^{i} \big( [i+1]^{-a} - [i]^{-a} \big) \Big),
\end{equation*}
so
\begin{equation*}
f(a,t) = \prod_{j=1}^{\infty} \frac{1+t^{j}}{1-t^{j}}  \times \Big( 1 +  \sum_{i \geq 1} (-1)^{i} \big( [i+1]^{a} - [i]^{a} \big) \Big)
\end{equation*}
as desired.
\end{proof}
We note that 
\begin{equation*}
f(a,t) = \prod_{j=1}^{\infty} \frac{1+t^{j}}{1-t^{j}}  \times \Bigg(1 - ( [2]^{a} - 1) + ([3]^{a} - [2]^{a}) - ([4]^{a} - [3]^{a}) + \cdots\Bigg). 
\end{equation*}

We define the function 
\begin{equation*}
f_{s}(a,t) := \prod_{j=1}^{\infty} \frac{1-st^{j}}{1-t^{j}}  \times \Big( 1 + \sum_{i \geq 1} s^{i}\big([i+1]^{a} - [i]^{a}\big)\Big).
\end{equation*}
For $s=1$, taking the truncated sum from $i=1$ to $i=N-1$ we obtain $[N]^{a}$, so as $N \rightarrow \infty$, we obtain
\begin{equation*}
f_{1}(a,t) = \lim_{N \rightarrow \infty} \Big(\frac{1-t^{N}}{1-t}\Big)^{a} = \frac{1}{(1-t)^{a}},
\end{equation*}
which is the usual Hilbert series.
For $s=-1$, we recover the formula for $f(a,t)$ in the previous theorem.

\section{Examples}
In this section, we use the results of the previous sections to compute some formulas for the signature character $ch_{s}(M_{c}(\lambda))$ for small values of $n$, and $\lambda$ with $|\lambda| = n$.  In particular, we give formulas for the signature character of $M_{c}(\lambda)$ for $c<0$ in Examples 1-4 and $-1<c<0$ in Examples 5-10.  Note that this is equal to the signature character of $M_{-c}(\lambda')$, where $\lambda'$ is the conjugate partition of $\lambda$.  

\textbf{Example 1.} Let $n=2$ and $\lambda = (2)$.

For $c < 0$, we obtain $\frac{1}{(1-t)^{2}}$, which agrees with the fact that $M_{c}(\text{triv})$ is unitary for $c < 0$.

\textbf{Example 2.} Let $n=2$ and $\lambda = (1 1)$.

We let $I_{0} = [-1/2, 1/2]$, and $I_{m} = [-1/2-m, 1/2 - m]$ for $m \in \mathbb{Z}_{+}$, then we obtain the following formula for the signature character:
\begin{equation*}
\frac{1-2t(1-t^{2} + t^{4} - t^{6} + \cdots + (-1)^{m-1}t^{2(m-1)})}{(1-t)^{2}} = \frac{1-2t(1-(-t^{2})^{m})/(1+t^{2})}{(1-t)^{2}}.
\end{equation*}

\textbf{Example 3.} Let $n=3$ and $\lambda = (3)$.

For $c <0$, we obtain $\frac{1}{(1-t)^{3}}$, which agrees with the fact that $M_{c}(\text{triv})$ is unitary for $c < 0$.

\textbf{Example 4.} Let $n=3$ and $\lambda = (2 1)$.

We first consider the interval type $[-2/3, -1/3]$.  For the intervals $[-2/3 - m, -1/3 - m]$, we get the following formula:
\begin{equation*}
\frac{2(1-t) + 2t^{2}(t^{2}-1)(1-(-t^{3})^{m})/(1+t^{3})}{(1-t)^{3}}.
\end{equation*}

The other type of interval is $[-4/3, -2/3]$.  For the intervals $[-4/3 - m, -2/3 - m]$ we get the following formula:
\begin{equation*}
\frac{2(1-t-t^{2}) + 2t^{4}(t+1)(1-(-t^{3})^{m})/(1+t^{3})}{(1-t)^{3}}.
\end{equation*}

\textbf{Example 5.} Let $n=3$ and $\lambda = (1 1 1)$.  

For the interval $(-1/3, 0)$ we obtain $\frac{1}{(1-t)^{3}}$.  For the interval $(-1/2, -1/3)$ we obtain $\frac{1-4t + 2t^{2}}{(1-t)^{3}}$.  For the interval $(-2/3, -1/2)$ we obtain $\frac{1-4t + 2t^{2} + 2t^{3}}{(1-t)^{3}}$.  For the interval $(-1, -2/3)$ we obtain 
\begin{equation*}
\frac{1-4t + 6t^{2} - 2t^{3} - 2t^{4}}{(1-t)^{3}}.
\end{equation*}

\textbf{Example 6.}  Let $n=4$ and $\lambda = (4)$.  For $c<0$ we obtain $\frac{1}{(1-t)^{4}}$.

\textbf{Example 7.} Let $n=4$ and $\lambda = (3 1)$.  For the interval $(-1/4, 0)$ we obtain $\frac{3}{(1-t)^{4}}$.  For the interval $(-1/2, -1/4)$ we obtain $\frac{3-2t}{(1-t)^{4}}$.  For the interval $(-3/4, -1/2)$ we obtain $\frac{3-2t - 2t^{2}}{(1-t)^{4}}$.  For the interval $(-1, -3/4)$ we obtain $\frac{3-2t-2t^{2}-2t^{3}}{(1-t)^{4}}$.

\textbf{Example 8.} Let $n=4$ and $\lambda = (2 2)$.  For the interval $(-1/3, 0)$ we obtain $\frac{2}{(1-t)^{4}}$.  For the interval $(-1/2, -1/3)$ we obtain $\frac{2-2t^{2}}{(1-t)^{4}}$.  For the interval $(-2/3, -1/2)$ we obtain $\frac{2-6t + 2t^{2} + 2t^{3}}{(1-t)^{4}}$.  For the interval $(-1, -2/3)$ we obtain 
\begin{equation*}
\frac{2-6t + 2t^{2} + 2t^{3} + 2t^{4}}{(1-t)^{4}}.  
\end{equation*}

\textbf{Example 9.} Let $n=4$ and $\lambda = (2 1 1)$.  For the interval $(-1/4, 0)$ we obtain $\frac{3}{(1-t)^{4}}$.  For the interval $(-1/2, -1/4)$ we obtain $\frac{3-6t + 2t^{2}}{(1-t)^{4}}$.  For the interval $(-3/4, -1/2)$ we obtain $\frac{3-10t + 8t^{2}}{(1-t)^{4}}$.  For the interval $(-1, -3/4)$ we obtain 
\begin{equation*}
\frac{3-10t + 8t^{2} + 6t^{3} - 4t^{4} - 4t^{5} - 2t^{6}}{(1-t)^{4}}.  
\end{equation*}

\textbf{Example 10.} Let $n=4$ and $\lambda = (1 1 1 1)$.  For the interval $(-1/4, 0)$ we obtain $\frac{1}{(1-t)^{4}}$.  For the interval $(-1/3, -1/4)$ we obtain $\frac{1-6t + 6t^{2} - 2t^{3}}{(1-t)^{4}}$.  For the interval $(-1/2, -1/3)$ we obtain 
\begin{equation*}
\frac{1-6t + 10t^{2} - 2t^{3}-2t^{4}}{(1-t)^{4}}.  
\end{equation*}
For the interval $(-2/3, -1/2)$ we obtain 
\begin{equation*}
\frac{1-6t + 16t^{2} - 18t^{3}+2t^{4} + 4t^{5} + 2t^{6}}{(1-t)^{4}}.  
\end{equation*}
For the interval $(-3/4, -2/3)$ we obtain 
\begin{equation*}
\frac{1-6t+16t^{2}-18t^{3}-2t^{4}+16t^{5}-2t^{6}-4t^{7}-2t^{8}}{(1-t)^{4}}.  
\end{equation*}
For the interval $(-1, -3/4)$ we obtain 
\begin{equation*}
\frac{1-6t+16t^{2}-24t^{3}+18t^{4}-8t^{6}+2t^{8}+2t^{9}}{(1-t)^{4}}. 
\end{equation*}

\end{document}